\documentclass{amsart}

\usepackage[utf8]{inputenc}
\usepackage[T1]{fontenc}
\usepackage{verbatim}
\usepackage{graphicx}
\usepackage{graphicx,caption2,psfrag,float,color}
\usepackage{amssymb}
\usepackage{amscd}
\usepackage{amsmath}

\usepackage[T1]{fontenc}

\newtheorem{theorem}{Theorem}[section]

\newtheorem{cor}[theorem]{Corollary}
\newtheorem{lemma}[theorem]{Lemma}

\numberwithin{equation}{subsection}
\newtheorem{definition}[theorem]{Definition}

\title{Large gaps between consecutive prime numbers containing perfect $k$-th
powers of prime numbers}
\author{Helmut Maier and Michael Th. Rassias}
\date{\today}
\address{Department of Mathematics, University of Ulm, Helmholtzstrasse 18, 89081 Ulm, Germany.}
\email{helmut.maier@uni-ulm.de}
\address{Institute of Mathematics, University of Z\"urich, Winterthurerstrasse 190, CH-8057 Zürich, Switzerland  \& Institute for Advanced Study, Einstein Drive, Princeton, New Jersey 08540 USA}
\email{michail.rassias@math.uzh.ch, michailrassias@math.princeton.edu}
\thanks{}

\begin{document}

 \maketitle
 
\begin{abstract} Let $k\geq 2$ be a fixed natural number. We establish the existence of infinitely many pairs of consecutive primes $p_n$, $p_{n+1}$ satisfying
$$ p_{n+1}-p_n\geq  c\:\frac{\log p_n\: \log_2 p_n\: \log_4 p_n}{\log_3 p_n}\:,$$
with $c$ being a fixed positive constant,  for which the interval $(p_n, p_{n+1})$ contains
the $k$-th power of a prime number.\\


\end{abstract}

\section{Introduction and Statement of Main Theorem}
In their paper \cite{konyagin}, K. Ford, D. R. Heath-Brown and S. Konyagin prove the 
existence of infinitely many ``prime-avoiding" perfect $k$-th powers for any positive integer $k$.\\
They give the following definition of \textbf{prime avoidance}: an integer $m$ is called prime avoiding with constant $c$ if $m+u$ is composite for all integers $u$ satisfying
$$|u|\leq c\:g_1(m)\:,$$
with $c$ being a positive constant and 
$$g_1(m)=\frac{\log m\log_2m\log_4m}{(\log_3m)^2}\:.$$
Here $\log_k x:=\log(\log_{k-1}x)$.\\
In \cite{maierrassias} the authors of the present paper extended this result by proving the existence of infinitely many prime-avoiding $k$-th powers of prime numbers.\\
Their method of proof consists in a combination of the method of \cite{konyagin} with
the matrix method of the first author \cite{maier}. The matrix $\mathcal{M}$ employed in this technique is of the form:
$$\mathcal{M}=(a_{r,u})_{\substack{1\leq r\leq P(x)^{D-1}\\ u\in\mathcal{B}}},$$
with $P(x)$ being a product of many small prime numbers and $D$ is a fixed positive integer, where the \textit{rows} 
$${R}(r)=\{a_{r,u}\::\: u\in\mathcal{B}  \}$$
of the matrix are \textit{translates} -- in closer or wider sense -- of the \textit{base-row} $\mathcal{B}$. Moreover, $\mathcal{B}$ is contained in an interval of consecutive integers and
consists of the few integers that are coprime to $P(x)$.\\
The columns of the matrix $\mathcal{M}$ are arithmetic progressions (or -- in the case of \cite{maierrassias} -- shifted powers of elements of arithmetic progressions).
The appearance of primes in these arithmetic progressions can be studied using
results on primes in arithmetic progressions.\\
The construction of the base-row $\mathcal{B}$ in its simplest form has been carried out by 
Paul Erd\H{o}s \cite{erdos}
and R. A. Rankin \cite{rankin} in their papers on large gaps between consecutive primes. They obtain the following result: Infinitely often
$$p_{n+1}-p_n\geq c\:g_1(n)\:,$$
where $c>0$ is a fixed constant. Until recently all improvements only concerned the constant $c$ (cf. \cite{maierpom}, \cite{pintz}, \cite{rankin2}).\\
In the papers \cite{ford}, \cite{ford2} and \cite{maynard} finally the function $g_1$ has been replaced by a function of a higher order of magnitude, solving a longstanding problem of Erd\H{o}s.\\
K. Ford, B. J. Green, S. Konyagin, J. Maynard and T. Tao \cite{ford2} have proved the following result:
$$p_{n+1}-p_n\geq c\: g_2(n)$$
infinitely often, where
$$g_2(m)=\frac{\log m\:\log_2 m\:\log_4 m}{\log_3 m}\:.$$
In this paper we combine the methods of the papers \cite{konyagin}, \cite{ford2} and  \cite{maynard} to
obtain the following theorem.
\begin{theorem}\label{thm1}
There is a constant $c>0$ and infinitely many $n$, such that
$$p_{n+1}-p_n\geq c\:\frac{\log p_n\:\log_2 p_n\: \log_4 p_n}{\log_3 p_n}$$
and the interval $[p_n,p_{n+1}]$ contains the $k$-th power of a prime.
\end{theorem}
\noindent\textit{Remark.} With a minor modification we could also obtain perfect powers of primes that fulfill the prime-avoidance property, which means that the prime power is isolated on both sides. For the sake of simplicity we restrict ourselves to one-sided isolation. 
\section{Basic Structure of the Proof}
The main part of the proof consists of the construction of the base-row of our matrix. It will closely follow the construction carried out in \cite{ford2}.\\
We repeat the description of the basic structure of the proof given on page 5 of \cite{ford2}. The chain of implications is 
\[
\substack{\mathsection\ 6 \\ \text{Thm 5}\\ \mathsection\ 7.8}\ \ \ \substack{\Longrightarrow\\ \mathsection\ 6}\ \ \ 
\substack{\mathsection\ 4 \\ \text{Thm 4}}\ \ \ \substack{\Longrightarrow\\ \mathsection\ 4.5}\ \ \  \substack{\mathsection\ 3 \\ \text{Thm 2}}\ \ \ 
 \substack{\Longrightarrow\\ \mathsection\ 3}\ \ \  \text{Thm 1}\tag{2.1}
\]
The chain of implications is indicated in the bottom row whereas the number of the section of the statement is followed in the top row.\\
The construction of the ``base-row'' of our matrix, which will be described in \mbox{Theorem \ref{thm2},} will have the same structure. Each of the Theorems 1 to 5 of \cite{ford2} whose mutual relations are displayed in (2.1) will be replaced by a suitable modification. 
\section{Sieving a Set of Primes}
Given a large real number $x$, define
\[
y:=cx\:\frac{\log x\: \log_3 x}{\log_2 x}\:,\tag{3.1}
\]
where $c$ is a fixed positive constant. Let
\[
z:=x^{\log_3 x/(4\log_2 x)}\tag{3.2}
\]
and introduce the three disjoint sets of primes
\[
S:=\{s\ \text{prime}\::\: \log^{20}x<s\leq z  \}\:,\tag{3.3}
\]
\[
P:=\{p\ \text{prime}\::\:x/2<p\leq x  \}\:,\tag{3.4}
\]
\[
Q:=\{q\ \text{prime}\::\:x<q\leq y  \}\:.\tag{3.5}
\]
For residue classes $\vec{a}=(a_s\:\bmod\: s)_{s\in S}$ and $\vec{b}=(b_p\:\bmod\: p)_{p\in P}$ define the sifted sets
$$S(\vec{a}):=\{ n\in\mathbb{Z}\::\: n\not\equiv a_s(\bmod\:s)\ \text{for all}\ s\in S\}$$
and likewise
$$S(\vec{b}):=\{ n\in\mathbb{Z}\::\: n\not\equiv b_p(\bmod\:p)\ \text{for all}\ p\in P\}\:.$$
In modification of \cite{ford2} we now restrict the residue-classes used for sieving as follows:\\
We set
$$\mathcal{A}:=\{\vec{a}=(a_s\:\bmod\: s)_{s\in S}\::\: \exists\: c_s\ \text{such that}\ a_s\equiv\:1-(c_s+1)^k\:(\bmod\: s),\ c_s\not\equiv\:-1\:(\bmod\: s)\}$$
$$\mathcal{B}:=\{\vec{b}=(b_p\:\bmod\: p)_{p\in P}\::\: \exists\: d_p\ \text{such that}\ b_p\equiv\:1-(d_p+1)^k\:(\bmod\: p),\ d_p\not\equiv\:-1\:(\bmod\: p\}\:.$$
\begin{theorem}\label{thm2}(Sieving primes)
Let $x$ be sufficiently large and suppose that $y$ obeys (3.1). Then there are vectors $\vec{a}\in\mathcal{A}$ and $\vec{b}\in\mathcal{B}$, such that
\[
\#\{Q\cap S(\vec{a})\cap S(\vec{b}) \}\ll\frac{x}{\log x}\:.\tag{3.6}
\]
\end{theorem}
Here, we show why Theorem \ref{thm2} implies Theorem \ref{thm1}. We shall prove \mbox{Theorem \ref{thm2}} in subsequent sections.\\
Let $\vec{a}$ and $\vec{b}$ be as in Theorem \ref{thm2}. We extend the tuplet $\vec{a}$ to a tuplet $(a_p)_{p\leq x}$ of congruence classes $a_p(\bmod\:p)$ for all
primes $p\leq x$ by setting $a_p:=b_p$ for $p\in P$ and $a_p:=0$ for $p\not\in S\cup P$, and consider the sifted set
$$\mathcal{T}:=\{n\in[y]\setminus [x]\::\: n\not\equiv a_p\:(\bmod\:p)\ \text{for all}\ p\leq x  \}\:,$$
where $[y]$, $[x]$ denote the intervals $[0,y]$,  $[0,x]$ respectively.\\
The elements of $\mathcal{T}$, by construction, are not divisible by any prime in
$(0, \log^{20} x]$ or in $(z, x/2]$. Thus, each element must either be a $z$-smooth
number (i.e., a number with all prime factors at most $z$), or must consist of a prime greater than $x/2$, possibly multiplied by some additional primes that are all at least $\log^{20} x$. However, $y=o(x\log x)$. Thus, we see that an element of $\mathcal{T}$ is either a $z$-smooth number or a prime in $Q$. In the second case the element lies in
$$Q\cap S(\vec{a})\cap S(\vec{b})\:.$$
Conversely, every element of $Q\cap S(\vec{a})\cap S(\vec{b})$ lies in $\mathcal{T}$. Thus, $\mathcal{T}$ only differs from $Q\cap S(\vec{a})\cap S(\vec{b})$ by a set
$\mathcal{U}$ consisting of $z$-smooth numbers in $[y]$.\\
To estimate $\#\mathcal{U}$ let
$$u:=\frac{\log y}{\log z}\:,$$
so from (3.1), (3.2) one has
$$u\sim4\:\frac{\log_2 x}{\log_3 x}\:.$$
By standard counts for smooth numbers (e.g. De Bruijn's theorem \cite{bruijn}) and 
(3.5) we thus have
\begin{align*}
\#\mathcal{U}&\ll y\exp\left(-u\log u+O(u\log(\log(u+2))) \right)\\
&=\frac{y}{\log^{4+o(1)} x}=o\left( \frac{x}{\log x}\right).
\end{align*}
Thus, we find
\begin{lemma}\label{lem31}
$$\#\mathcal{T}\ll\frac{x}{\log x}\:.$$
\end{lemma}
We now further reduce the sifted set $\mathcal{T}$ by using the prime numbers
from the interval $(x, C_0x]$, $C_0>1$ being a sufficiently large constant.\\
Here we follow -- with some modification in the notation -- the papers \cite{konyagin},
\cite{maierrassias}.\\
We distinguish the cases $k$ odd and $k$ even.
\begin{definition}
Let
\begin{eqnarray}
\nonumber\\
\tilde{P}&=&\left\{ 
  \begin{array}{l l}
    p\::\:\:  &x<p\leq C_0x,\ p\equiv2\:(\bmod\:3), \text{if $k$ is odd}\vspace{2mm}\\ 
    p\::\:\:  &x<p\leq C_0x,\ p\equiv3\:(\bmod\:2k), \text{if $k$ is even.}\vspace{2mm}\\ 
  \end{array} \right.
\nonumber
\end{eqnarray}
For $k$ even and $\delta>0$, we set
$$\mathcal{U}=\left\{ u\in[0,y]\::\: \left( \frac{-u}{p}\right)=1\ \text{for at most}\ \frac{\delta x}{\log x}\ \text{primes}\ p\in\tilde{P} \right\}\:.$$
\end{definition}
\begin{lemma}\label{lem32}
We have 
$$\#\mathcal{U}\ll_{\varepsilon}x^{1/2+2\varepsilon}\:.$$
\end{lemma}
\begin{proof}
This is formula (4) of \cite{konyagin}.
\end{proof}
\begin{lemma}\label{lem33}
There are pairs $(u,p_u)$ with $u\in\mathcal{T}$, $p_u\in\tilde{P}$, such that all $u\in\mathcal{T}$ satisfy a congruence
\[
u\equiv 1-(e_u+1)^k\:(\bmod\: p_u)\ \text{where}\ e_u\not\equiv-1\:(\bmod\: p_u)\tag{3.7}
\]
with the possible exceptions of $u$ from an exceptional set $V$ with
$$\# V\ll x^{1/2+2\varepsilon}\:.$$
\end{lemma}
\begin{proof}
If $k$ is odd, the congruence (3.7) is solvable, whenever $p\equiv 2\:(\bmod\:3)$.\\
If $k$ is even, the congruence is solvable whenever $p\equiv 3\:(\bmod\:2k)$ and
$\left(\frac{-u}{p} \right)=1$. The claim now follows from Lemma \ref{lem31} and
Lemma \ref{lem32}.
\end{proof}
We now conclude the deduction of Theorem \ref{thm1} by the application of the matrix method. The following definition is borrowed from \cite{maier}.
\begin{definition}
Let us call an integer $q>1$ a ``good'' modulus, if $L(s,\chi)\neq 0$ for all characters $\chi\: \bmod\:q$ and all $s=\sigma+it$ with 
$$\sigma >1-\frac{C_2}{\log(q(|t|+1))} \:.$$
\end{definition}
This definition depends on the size of $C_2>0$.
\begin{lemma}\label{lem34}
There is a constant $C_2>0$, such that, in terms of $C_2$, there exist arbitrary large
values of $x$, for which the modulus
$$P(x)=\prod_{p<x} p$$
is good.
\end{lemma}
\begin{proof}
This is Lemma 1 of \cite{maierpom}. 
\end{proof}
\begin{lemma}\label{lem35}
Let $q$ be a good modulus. Then
$$\pi(x;q,a)\gg\frac{x}{\phi(q)\log x}\:,$$
where $\phi(\cdot)$ denotes Euler's totient function,
uniformly for $(a,q)=1$ and $x\geq q^D$. Here the constant $D$ depends only on the value of $C_2$ in Lemma \ref{lem34}.
\end{lemma}
\begin{proof}
This result, which is due to Gallagher \cite{gallagher}, is Lemma 2 from \cite{maier}.
\end{proof}
We now define the matrix $\mathcal{M}$.
\begin{definition}
Choose $x$, such that $P(C_0x)$ is a good modulus. Let $\vec{a}\in\mathcal{A}$
and $\vec{b}\in\mathcal{B}$ be given. From the definition of $\mathcal{A}$ and $\mathcal{B}$, there are 
$$(c_s\: \bmod\: s)_{s\in S}\ \ \text{and}\ \ (d_p\: \bmod\: p)_{p\in P}\:,\ c_s\not\equiv -1\:(\bmod\: s),\: d_p\not\equiv -1\:(\bmod\: p)\:, $$
such that
\begin{align*}
\vec{a}=(1-(c_s+1)^k\: \bmod\:s)_{s\in S}\ \ \ \text{and}\ \ \ \vec{b}=(1-(d_p+1)^k\: \bmod\:p)_{p\in P}\:.
\end{align*}
We now determine $m_0$ by 
$$1\leq m_0<P(C_0x)$$
and the congruences 
\begin{align*}
m_0&\equiv\: c_s\:(\bmod\: s)\\
m_0&\equiv\: d_p\:(\bmod\: p)\\
m_0&\equiv\: 0\:(\bmod\: q),\ \ q\in(1,x],\ q\not\in S\cup P\tag{3.8}\\
m_0&\equiv\: e_u\:(\bmod\: p_u),\ \ (e_u,p_u)\ \text{given by (3.7)}\\
m_0&\equiv\: g_p\:(\bmod\: p),\ \ \text{for all other primes}\ p\leq C_0x,\ g_p\ \text{arbitrary}.
\end{align*}
\end{definition}

\noindent By the Chinese Remainder Theorem $m_0$ is uniquely determined.\\ 
We let
$$\mathcal{M}=(a_{r,u})_{\substack{1\leq r\leq P(x)^{D-1}\\ 1\leq u\leq y}}$$
with
$$a_{r,u}=(m_0+1+rP(x))^k+u-1\:.$$
For $1\leq r\leq P(x)^{D-1}$, we denote by 
$${R}(r)=(a_{r,u})_{0\leq u\leq y}$$
the $r$-th row of $\mathcal{M}$ and for $0\leq u\leq y$, we denote by
$$C(u)=(a_{r,u})_{1\leq r\leq P(x)^{D-1}}$$
the $u$-th column of $\mathcal{M}$.
\begin{lemma}\label{lem36}
We have that $a_{r,u}$, $2\leq u\leq y$, is composite unless $u\in V$.
\end{lemma}
\begin{proof}
From the congruences $m_0\equiv c_s\:(\bmod\: s)$, resp. $m_0\equiv d_p\:(\bmod\: p)$, \mbox{$m_0\equiv 0\:(\bmod\: q)$,} $m_0\equiv e_u\:(\bmod\: p_u)$ in (3.8), it follows that for  $u\equiv 1-(c_s+1)^k\:(\bmod\: s)$, resp.  \mbox{$u\equiv 1-(d_p+1)^k\:(\bmod\: p)$,} 
 $u\equiv 0\:(\bmod\: q)$,   $u\equiv 1-(e_u+1)^k\:(\bmod\: p_u)$, we have
 $a_{r,u}\equiv 0\:(\bmod\: s)$, resp.  \mbox{$a_{r,u}\equiv 0\:(\bmod\: p),$ }
 $a_{r,u}\equiv 0\:(\bmod\: q)$,  $a_{r,u}\equiv 0\:(\bmod\: p_u).$
 \end{proof}
\begin{definition}
Let 
$$\mathcal{R}_0(\mathcal{M})=\{ r\::\: 1\leq r\leq P(x)^{D-1},\ m_0+1+rP(x)\ \text{is prime} \}\:,$$
$$\mathcal{R}_1(\mathcal{M})=\{ r\::\: 1\leq r\leq P(x)^{D-1},\ r\in\mathcal{R}_0(\mathcal{M}),\ \text{R}(r)\ \text{contains a prime number} \}\:.$$
\end{definition} 
\noindent\textit{Remark.} We observe that each row ${R}(r)$ with $r\in\mathcal{R}_0(\mathcal{M})$ has as its first element
$$a_{r,1}=(m_0+1+rP(x))^k\:,$$
the $k$-th power of the prime $m_0+1+rP(x)$.\\
If $r\in \mathcal{R}_0(\mathcal{M})\setminus \mathcal{R}_1(\mathcal{M})$, $a_{r,1}$ is the $k$-th power of a prime of the desired kind.\\
To deduce Theorem \ref{thm1} from Theorem \ref{thm2} it thus remains to show that 
\mbox{$\mathcal{R}_0(\mathcal{M})\setminus \mathcal{R}_1(\mathcal{M})$} is non-empty.
\begin{lemma}\label{lem37}
We have 
$$\#\mathcal{R}_0(\mathcal{M})\gg\frac{P(x)^D}{\phi(P(x))\log \left(P(x)^D\right)}\:.$$
\end{lemma}
\begin{proof}
This follows from Lemma \ref{lem35}.
\end{proof}
We obtain an upper estimate for $\mathcal{R}_1(\mathcal{M})$ by the observation that, if
${R}(r)$ contains 
a prime number, then $m_0+1+rP(x)$ and $(m_0+1+rP(x))^k+v-1$ are primes for some $v\in V$.\\
The number 
$$t(v)=\#\{r\::\: 1\leq r\leq P(x)^{D-1},\ m_0+1+rP(x)\ \text{and}\ (m_0+1+rP(x))^k+v-1\ \text{are primes}  \}$$
is estimated by standard sieves as in Lemma 6.1 of \cite{maierrassias}.\\
This concludes the deduction of Theorem \ref{thm1} from Theorem \ref{thm2}.
\section{Using a Hypergraph Covering Theorem}
In the previous section we reduced matters to obtaining residue classes $\vec{a}\in\mathcal{A}$, $\vec{b}\in\mathcal{B}$, such that the sifted set $Q\cap S(\vec{a})\cap S(\vec{b})$ is small.\\
In \cite{ford2}, where it is not requested that $\vec{a}\in\mathcal{A}$, $\vec{b}\in\mathcal{B}$, this problem is handled by using a hypergraph covering theorem -- of 
a purely combinatorial nature -- generalizing a result of Pippenger and Spencer \cite{pip}.\\
Here no modification of the result of \cite{ford2} is necessary, but we use exactly the same hypergraph covering theorem, Theorem \ref{thm3} in our approach:
\begin{theorem}\label{thm3} (see Theorem 3 of \cite{ford2}) (Probabilistic covering). There exists a constant $C_0\geq 1$ such that the following holds. Let $D,r,A\geq 1$, $0<\kappa\leq 1/2$, and let
$m\geq 0$ be an integer. Let $\delta>0$ satisfy the smallness bound
\[
\delta\leq \left( \frac{\kappa^A}{C_0\exp(AD)} \right)^{10^{m+2}}\:. 
\]
Let $I_1,\dots, I_m$ be disjoint finite non-empty sets, and let $V$
be a finite set. For each $1\leq j\leq m$ and $i\in I_j$, let $\bold{e}_i$
be a random finite subset of $V$. Assume the following:
\begin{itemize}
\item (Edges not too large) Almost surely for all $j=1,\ldots, m$ and $i\in I_j$, we have
$$\#\bold{e}_i\leq r_i;$$
\item (Each sieve step is sparse) For all $j=1,\ldots, m$, $i\in I_j$ and $v\in V$,
$$\mathbb{P}(v\in \bold{e}_i) \leq \frac{\delta}{(\# I_j)^{1/2}};$$
\item (Very small codegrees) For every $j=1,\ldots, m$ and distinct $v_1, v_2\in V$,
$$\sum_{i\in I_j}\mathbb{P}(v_1, v_2\in\bold{e}_i)\leq \delta$$
\item (Degree bound) If for every $v\in V$ and $j=1,\ldots, m$ we introduce the normalized degrees 
$$d_{I_j}(v):=\sum_{i\in I_j}\mathbb{P}(v\in \bold{e}_i)$$
and then recursively define the quantities $P_j(v)$ for $j=0,\ldots, m$ and $v\in V$
by setting
$$P_0(v):=1$$
and
$$P_{j+1}(v):=P_j(v)\exp(-d_{I_{j+1}}(v)/P_j(v))$$
for $j=0,\ldots, m-1$ and $v\in V$, then we have
$$d_{I_j}(v)\leq DP_{j-1}(v)\ \ (1\leq j\leq m,v\in V)$$
and
$$P_j(v)\geq k\ \ (0\leq j\leq m,v\in V)\:.$$
\end{itemize}
Then we can find random variables $\bold{e'}_i$ for each $i\in \bigcup_{j=1}^m I_j$
with the following properties:\\
\ \ (a) For each $i\in \bigcup_{j=1}^m I_j$, the essential support of $\bold{e}'_i$ is
contained in the essential support of $\bold{e}_i$, union the empty set singleton $\{\emptyset\}$. In other words, almost surely $\bold{e}'_i$ is either empty, or is a set that
$\bold{e}_i$ also attains with positive probability.\\
\ \ (b) For any $0\leq J\leq m$ and any finite subset $e$ of $V$ with $\#e\leq A-2rJ$,
one has 
$$\mathbb{P}\left(e\subset V\setminus \bigcup_{j=1}^J \bigcup_{i\in I_j}\bold{e}'_i \right)=\left( 1+O_{\leq}(\delta^{1/10^{J+1}}) \right)P_J(e)$$
where 
$$P_{j}(e):=\prod_{v\in e}P_j(v).$$
\end{theorem}
The proof of our result is actually based on a Corollary of Theorem \ref{thm3}, which we quote exactly from \cite{ford2}.
\begin{cor}\label{col3}(see Corollary 3 of \cite{ford2})\\
Let $x\rightarrow \infty$. Let $\mathcal{P}'$, $\mathcal{Q}'$ be sets of primes with 
$\#\mathcal{P}'\leq x$ and $\#\mathcal{Q}'>(\log_2 x)^3$. For each $p\in \mathcal{P}'$, let $\bold{e}_p$ be a random subset of $\mathcal{Q}'$ satisfying the size bound
$$\#\bold{e}_p\leq r=O\left(\frac{\log x\log_2 x}{\log_2^2 x}  \right)\ \ (p\in \mathcal{P}').$$
Assume the following:
\begin{itemize}
\item (Sparsity) For all $p\in\mathcal{P}'$ and $q\in\mathcal{Q}'$,
$$\mathbb{P}(q\in\bold{e}_p)\leq x^{-1/2-1/10}.$$
\item (Uniform covering) For all but at most $$\frac{1}{(\log_2 x)^2}\#\mathcal{Q}'$$
elements $q\in\mathcal{Q}'$, we have
\[
\sum_{p\in \mathcal{P}'}\mathbb{P}(q\in\bold{e}_p)=C+O_{\leq}\left(\frac{1}{(\log_2 x)^2}\right)\tag{4.1}
\]
for some quantity $C$, independent of $q$, satisfying
$$\frac{5}{4}\log 5\leq C \ll 1\:.$$
\item (Small codegrees) For any distinct $q_1$, $q_2\in\mathcal{Q}'$,
\[
\sum_{p\in\mathcal{P}'}\mathbb{P}(q_1,q_2\in\bold{e}_p)\leq x^{-1/20}.\tag{4.2}
\]
Then for any positive integer $m$ with
$$m\leq \frac{\log_3 x}{\log 5}\:,$$
we can find random sets $\bold{e}'_p\subseteq\mathcal{Q}'$ for each $p\in\mathcal{P}'$ such that
$$\#\{ q\in\mathcal{Q}'\::\: q\not\in \bold{e}'_p\ \text{for all}\ p\in\mathcal{P}' \}\sim 5^{-m}\#\mathcal{Q'}$$
with probability $1-o(1)$. More generally, for any $\mathcal{Q}''\subset \mathcal{Q}'$
with cardinality at least $(\#\mathcal{Q'})/\sqrt{\log_2 x},$ one has 
$$\#\{q\in\mathcal{Q}''\::\: q\not\in\bold{e}'_p\ \text{for all}\ p\in\mathcal{P}'  \}\sim 5^{-m}\#\mathcal{Q}''$$
with probability $1-o(1)$. The decay rates in the $o(1)$ and $\sim$ notation are uniform in $\mathcal{P}'$, $\mathcal{Q}'$, $\mathcal{Q}''$.
\end{itemize}
\end{cor}
Using this corollary we may now reduce Theorem \ref{thm2} to the following.
\begin{theorem}\label{thm4}(see Theorem 4 of  \cite{ford2}) (Random construction).
Let $x$ be a sufficiently large real number and define $y$ by (3.1). Then there is a quantity $C$ with
\[
C\asymp\frac{1}{c}\tag{4.3}
\]
with the implied constants independent of $c$, a tuple of positive integers
$(h_1,\ldots,h_r)$ with $r\leq \sqrt{\log x}$, and some way to choose random vectors 
$\vec{\bold{a}}=\: (\bold{a}_s\:\bmod\: s)_{s\in\mathcal{S}}$ and ${\vec{\bold{n}}}=(\bold{n}_p)_{p\in\mathcal{P}}$ of congruence classes $\bold{a}_s\:\bmod\: s$
and integers $\bold{n}_p$ respectively, obeying the following:
\begin{itemize}
\item For every $\vec{a}$ in the essential range of $\vec{\bold{a}}$, one has
$$\mathbb{P}(q\in\bold{e}_p(\vec{a})\:|\:\vec{\bold{a}}=\vec{a})\leq x^{-1/2-1/10}\ \ (p\in\mathcal{P})\:,$$
where $\bold{e}_p(\vec{a}):=\{ \bold{n}_p+h_ip\::\: 1\leq i\leq r \}\cap \mathcal{Q}\cap S(\vec{a})$.
\item With probability $1-o(1)$ we have that 
\[
\#(\mathcal{Q}\cap S(\vec{\bold{a}}))\sim 80c\:\frac{x}{\log x}\: \log_2 x\:.\tag{4.4}
\]
\item Call an element $\vec{a}$ in the essential range of $\vec{\bold{a}}$ good if,
for all but at most $\frac{x}{\log x\log_2 x}$ elements $q\in\mathcal{Q}\cap S(\vec{a})$, one has
\[
\sum_{p\in\mathcal{P}}\mathbb{P}(q\in\bold{e}_p(\vec{a})\:|\:\vec{\bold{a}}=\vec{a})=C+O_{\leq}\left(\frac{1}{(\log_2 x)^2} \right)\:.\tag{4.5}
\]
Then $\vec{\bold{a}}$ is good with probability $1-o(1)$.
\end{itemize}
\end{theorem}
$$ $$
We now show (quoting \cite{ford2}) why Theorem \ref{thm4} implies Theorem \ref{thm2}. By (4.3) we may choose $0<c<1/2$ small enough so that (4.19) holds.
Take
$$m=\left\lfloor \frac{\log_3 x}{\log 5} \right\rfloor\:.$$
Now let $\vec{\bold{a}}$ and $\vec{\bold{n}}$ be the random vectors guaranteed by 
Theorem \ref{thm4}. Suppose that we are in the probability $1- o(1)$ event that $
\vec{\bold{a}}$ takes a value $\vec{a}$ which is good and such that (4.4) holds. Fix 
some $\vec{a}$ within this event. We may apply Corollary \ref{col3} with 
$\mathcal{P}'= \mathcal{P}$ and $\mathcal{Q}'= \mathcal{Q}\cup\mathcal{S}
(\vec{a})$ for the 
random variables $\bold{n}_p$ conditioned to $\vec{\bold{a}}=\vec{a}$. A few 
hypotheses of the Corollary must be verified. First, (4.1) follows from (4.35). The 
small codegree condition (4.2) is also quickly checked. Indeed, for distinct 
$q_1,q_2\in\mathcal{Q}'$, if $q_1,q_2\in\bold{e}_p(\vec{a})$ then $p\mid q_1- q_2$. 
But $q_1- q_2$ is a nonzero integer of size at most $x \log x$, and is thus divisible 
by at most one prime $p_0\in\mathcal{P}'$. Hence
$$\sum_{p\in\mathcal{P}'} \mathbb{P}(q_1,q_2\in\bold{e}_p(\vec{a}))=\mathbb{P}(q_1, q_2\in\bold{e}_{p_0}(\vec{a}))\leq x^{-1/2-1/10}\:,$$
the sum on the left side being zero if $p_0$ doesn't exist. By Corollary \ref{col3}, 
there exist random variables $\bold{e}'_p(\vec{a})$, whose essential range is 
contained in the essential range of $\bold{e}_p(\vec{a})$ together with $\emptyset$, 
and satisfying
$$\{q\in \mathcal{Q}\cap S(\vec{a})\::\: q\not\in\bold{e}'_p(\vec{a})\ \text{for all}\ p\in\mathcal{P}\}\sim 5^{-m}\#(\mathcal{Q}\cap S(\vec{a}))\ll\frac{x}{\log x}  $$
with probability $1 - o(1)$, where we have used (4.4). Since 
$$\bold{e}'_p(\vec{a})=\{\bold{n'}_p+h_ip\::\:1\leq i\leq r \}\cap\mathcal{Q}\cap S(\vec{a})$$ 
for some random integer $\bold{n}'_p$, it follows that
$$\{q\in \mathcal{Q}\cap S(\vec{a})\::\: q\not\equiv\bold{n}'_p\:(\bmod\: p)\ \text{for all}\ p\in\mathcal{P}\}\ll\frac{x}{\log x}$$
with probability $1 - o(1)$. Taking a specific $\vec{\bold{n'}}=\vec{n'}$ for which this 
relation holds and setting $b_p=n'_p$ for all $p$ concludes the proof of the claim
(3.6) and establishes Theorem \ref{thm2}.
\section{Proof of Covering Theorem}
The proof of Theorem \ref{thm3} is given in \cite{ford2}, Section 5.\qed
\section{Using a Sieve Weight}
We will establish Theorem \ref{thm4} by a probabilistic argument involving a certain weight function $w\::\:{P}\times\mathbb{Z}\rightarrow\mathbb{R}^+$, whose 
properties will be given in Theorem \ref{thm5}.\\
The construction of $w$ will be modeled on the construction of the function $w$ in 
\cite{ford2}, Theorem 5. The restrictions $\vec{a}\in\mathcal{A}$, $\vec{b}\in\mathcal{B}$ bring some additional complications. The function $w(p,n)$ will be different from zero only if $n$
belongs to a set $\mathcal{G}(p)$ of \textit{$p$-good integers}. The definition of
$\mathcal{G}(p)$ is based on the set $\mathcal{G}$ of \textit{good integers}.\\
We start with the definition of the set $\mathcal{G}$.
\begin{definition}\label{def61}
An admissible $r$-tuple is a tuple $(h_1,\ldots, h_r)$ of distinct integers that do not cover all residue classes modulo $p$ for any prime $p$.\\
For $(u,k)=1$ we define 
$$S_u:=\{s\::\: s\ \text{prime},\ s\equiv u(\bmod k),\ (\log x)^{20}<s\leq z\}$$
$$d(u)=(u-1,k),\ r^*(u)=\frac{1}{d(u)}\sum_{s\in S_u} s^{-1}.$$
For $n\in [x,y]$ let
$$r(n,u)=\sum_{\substack{s\in S_u\::\: \exists\ c_s\::\:  n\equiv 1-(c_s+1)^k(\bmod s)\\ c_s\not\equiv -1(\bmod\: s)}}s^{-1}$$
We set 
$$\mathcal{G}=\{n\::\: n\in [x,y],\ |r(n,u)-r^*(u)|\leq (\log x)^{-1/40}\ \text{for all}\ u\:(\bmod k),\ (u,k)=1\}.$$
For an admissible $r$-tuple to be specified later and for primes $p$ with $x/2<p\leq x$ we set 
$$\mathcal{G}(p)=\{n\in \mathcal{G}\::\: n+(h_i-h_l)p\in \mathcal{G},\ \forall\: i,l\leq r\}.$$
\end{definition}

\begin{theorem}\label{thm5}(Existence of good sieve weights)\\
Let $x$ be a sufficiently large real number and let $y$ be any quantity obeying (3.1).
Let $P$, $Q$ be defined by (3.4), (3.5). Let $r$ be a positive integer with
\[
r_0\leq r\leq \log^{\eta_0} x\tag{6.1}
\]
for some sufficiently large absolute constant $r_0$ and some sufficiently small $\eta_0>0$.\\
Let $(h_1,\ldots,h_r)$ be an 
admissible $r$-tuple contained in $[2r^2]$. Then one can find a positive quantity 
$$\tau\geq x^{-o(1)}$$
and a positive quantity $u=u(r)$ depending only on $r$ with 
\[
u\asymp \log r \:,\tag{6.2}
\]
and a non-negative function
$$w\::\: {P}\times\mathbb{Z}\rightarrow \mathbb{R}^{+}$$
supported on ${P}\times(\mathbb{Z}\cap[-y,y])$ with the following properties:
\[
w(p,n)=0,\ \ \text{unless}\ n\equiv1-(d_p+1)^k\:(\bmod\: p)\ \text{for some}\ d_p\in\mathbb{Z},\ d_p\not\equiv-1\:(\bmod\: p) \ \text{and}\ n\in\mathcal{G}(p)\:.\tag{6.3}
\]
Uniformly for every $p\in P$, one has
\[
\sum_{n\in\mathbb{Z}}w(p,n)=\left(1+O\left(\frac{1}{\log_2^{10}x} \right)\right)\left(\tau\frac{y}{\log x}\right)\:.  \tag{6.4}
\]
Uniformly for every $q\in\mathbb{Q}$ and $i=1,\ldots,r$, one has
\[
\sum_{p\in P}w(p,q-h_ip)=\left(1+O\left(\frac{1}{\log_2^{10}x}\right)\right)\tau\frac{u}{r}\frac{x}{2\log^r x}\:.\tag{6.5}
\]
Uniformly for every $h=O(y/x)$ that is not equal to any of the $h_i$ one has
\[
\sum_{q\in Q}\sum_{p\in P}w(p,q-hp)=O\left( \frac{1}{\log_2^{10}x}\:\tau\: \frac{xy}{\log^r x\: \log\log x} \right)\:. \tag{6.6}
\]
Uniformly for all $p\in P$ and $z\in\mathbb{Z}$
\[
w(p,n)=O\left(x^{1/3+o(1)}\right)\:.\tag{6.7}
\]
\end{theorem}
In this section we show how Theorem \ref{thm5} implies Theorem \ref{thm4}.\\
Let $x$, $c$, $y$, $z$, $S$, ${P}$, ${Q}$ be as in Theorem \ref{thm4}. We set 
$$r:=\lfloor (\log x)^{\eta_0}  \rfloor,\ \ \sigma:=\prod_{s\in S}\left(1-\frac{1}{s}\right)\:.$$
We now invoke Theorem \ref{thm5} to obtain quantities $\tau,u$ and weight
$w\::\: P\times\mathbb{Z}\rightarrow \mathbb{R}^+$ with the stated properties. \\
For each $p\in P$, let $\vec{\bold{n}}_p$ denote the random integer with probability density
\[
\mathbb{P}(\tilde{\bold{n}}_p=n):=\frac{w(p,n)}{\sum_{n'\in\mathbb{Z}}w(p,n')}\:,\tag{6.8}
\]
for all $n\in\mathbb{Z}$. From (6.4), (6.5) we have
\[
\sum_{p\in P}\mathbb{P}(q=\tilde{n}_p+h_ip)=\left(1+O\left(\frac{1}{\log_2^{10}x}\right)\right)\frac{u}{r}\frac{x}{2y},\tag{6.9}
\]
$(q\in Q,\ 1\leq i\leq r)$.
Also, from (6.4), (6.8), (6.2) one has 
$$\mathbb{P}(\tilde{\bold{n}}_p=n)\ll x^{-1/2-1/6+o(1)}\:,$$
for all $p\in P$ and $n\in\mathbb{Z}$.\\
We choose the random vector $\vec{\bold{a}}:=(a_s\:\bmod\:s)_{s\in S}$ by selecting each $a_s\:\bmod s$ uniformly at random from $\mathcal{A}_s$ independently in $s$.
\begin{lemma}\label{lem61}
Let $t\leq (\log x)^{3\eta_0}$ be a natural number, and let $n_1,\ldots, n_t$ be distinct integers from $\mathcal{G}$.Then, one has 
$$\mathbb{P}(n_1\ldots,n_t\in S(\vec{\bold{a}}))=\left(1+O\left(\frac{1}{(\log_2x)^{10}}\right)\right)\sigma^t\:.$$
\end{lemma}
\begin{proof}
For $\vec{n}=(n_1,\ldots, n_t)$, let $\mathcal{K}(\vec{n})$ be the set of $s\in S$ for which $s\mid n_l-n_i$, for $i\neq l$. Then, since 
$$n_l-n_i=O\left(x^{O(1)}\right),$$
we have
$$|\mathcal{K}(\vec{n})|=O((\log x)^3)\:.$$
Let $\vec{a}\in\mathcal{A}_s$, $1\leq u\leq k-1$, $(u,k)=1$. We write
$$\vec{a}_{u}=(a_{s_{1,u}},\ldots, a_{s_{r_u,u}} ),$$
where 
$$S_{u}\cap \mathcal{K}^c=\left\{ s_{1,u},\ldots, s_{r_u,u} \right\}\:.$$
We set
\begin{eqnarray}
\nonumber\\
\epsilon(h,s)&=&\left\{ 
  \begin{array}{l l}
    1, \ \ \  \text{if $n_h\equiv1-(c_{s,h}+1)^k(\bmod\: s)$ has a solution $c_{s,h}\not\equiv\:-1(\bmod\:s)$}\vspace{2mm}\\ 
    0, \ \ \  \text{otherwise}\vspace{2mm}\\ 
  \end{array} \right.
\nonumber
\end{eqnarray}
We have 
$$n_h\in S(\vec{a})\ \ \text{if and only if}\ n_h\in S(\vec{a}_{u}),\ \ \ \forall u\::\:(u,k)=1\:.$$
We now use certain well-known facts from the theory of $k$-th power residues.\\
There are 
$$\frac{s_{i,u}-1}{d(u)}-1$$ 
possible choices for the $a_{s_{i,u}}$. From these, for each $h$, $1\leq h\leq t$ there are $\epsilon(h,s_{i,u})$ choices such that
$$ a_{s_{i,u}} \equiv n_h\:(\bmod\: s_{i,u})\:.$$
Thus, the total number of choices for the $a_{s_{i,u}}$ for which not all $n_h\in S(\vec{a})$, \mbox{$(1\leq h\leq t)$} is
$$\sum_{h=1}^t\epsilon (h,s_{i,u})\:.$$
Since the choices for the components $a_s$ are independent, we have
\begin{align*}
\tag{6.10}&\mathbb{P}(n_1,\ldots,n_t\in S(\vec{\bold{a}}))\\
&=\prod_{u:(u,k)=1}\prod_{s\in S_{u}}\left(\frac{s-1-d(u)}{d(u)}\right)^{-1}\left(\frac{s-1-d(u)}{d(u)}-\sum_{h=1}^t\epsilon(h,s)\right)\left(1+O\left(\frac{(\log x)^3}{z_0} \right)\right)\\  
&=\prod_{u:(u,k)=1}\prod_{s\in S_{u}}\left( 1-d(u)s^{-1}\sum_{h=1}^t\epsilon(h,s)\right)\left(1+O(s^{-2})\right)\left(1+O((\log x)^{-17})\right)\:.
\end{align*}
We have 
$$\prod_{s\in S_{u}}\left( 1-d(u)s^{-1}\sum_{h=1}^t\epsilon(h,s)\right)=\exp\left(-\sum_{s\in S_{u}} d(u)s^{-1}\sum_{h=1}^t\epsilon(h,s)+O(s^{-2}) \right)\:.$$
Since $n_h\in\mathcal{G}$ for $1\leq h\leq t$, we have by the Definition \ref{def61}
for $\mathcal{G}$:
\[
\sum_{s\in S_{u}} s^{-1}\epsilon(h,s)=\frac{1}{d(u)}\sum_{s\in S_{u}}s^{-1}
+O\left(\left( \log x \right)^{-1/40}\right)\:. \tag{6.11}
\]
From (6.10) and (6.11), we thus obtain:
$$\mathbb{P}(n_1,\ldots, n_t\in S(\vec{\bold{a}}))=\left(1+O\left( \frac{1}{(\log x)^{1/40}}\right)  \right)\sigma^t\:.$$
\end{proof}
\begin{cor}\label{cor6464}
With probability $1-o(1)$ we have:
$$\#(Q\cap S(\vec{\bold{a}}))\sim\sigma \frac{y}{\log x}\sim 80c\frac{x}{\log x}\log_2 x\:.$$
\end{cor}
\begin{proof}
From Lemma \ref{lem61}, we have
$$\mathbb{E}\#(Q\cap S(\vec{\bold{a}}))=\left(1+O\left(\frac{1}{(\log_2 x)^5}\right)\right)\sigma\#Q\: $$
and
$$\mathbb{E}\#(Q\cap S(\vec{\bold{a}}))^2=\left(1+O\left(\frac{1}{(\log_2 x)^5}\right)\right)
\left(\sigma\#Q+\sigma^2(\#Q)(\#Q-1)\right)$$
and so by the prime number theorem we see that the random variable $\#(Q\cap S(\vec{a}))$ has mean
$$\left(1+O\left(\frac{1}{(\log_2 x)^5}\right)\right)\sigma\frac{y}{\log x} $$
and variance 
$$O\left(\frac{1}{(\log_2 x)^5}\left(\sigma\frac{y}{\log x}  \right)^2\right)\:.$$
The claim then follows from Chebyshev's inequality.
\end{proof}
For each $p\in P$, we consider the quantity
\[
X_p(\vec{a}):=\mathbb{P}\left(\tilde{\bold{n}}_p+h_ip\in S(\vec{a}),\ \text{for all}\ i=1,\ldots, r\right)\tag{6.12}
\]
and let $P(\vec{a})$ denote the set of all primes $p\in P$, such that
$$X_p(\vec{a})=\left(1+O\left(\frac{1}{(\log_2 x)^{10}} \right) \right)\sigma^r\:.$$
We now define the random variables $\bold{n}_p$ as follows. Suppose we are in the
event $\vec{\bold{a}}=\vec{a}$ for some $\vec{a}$ in the range of $\vec{\bold{a}}$.
If $p\in P\setminus P(\vec{a})$, we set $\bold{n}_p=0$. Otherwise, if $p\in P(\vec{a})$,
we define $\bold{n}_p$ to be the random integer with conditional probability distribution
$$\mathbb{P}\left(\bold{n}_p=n\:|\:\vec{\bold{a}}=\vec{a}\right):=\frac{Z_p(\vec{a};n)}{X_p(\vec{a})}\:,$$
where
$$Z_p(\vec{a};n):=1_{n+h_jp\in S(\vec{a}),\ \text{for}\ j=1,\ldots, r }\:\mathbb{P}(\tilde{\bold{n}}_p=n) $$
with the $\tilde{\bold{n}}_p$ jointly conditionally independent  on the event $\vec{\bold{a}}=\vec{a}$.
\begin{lemma}\label{lima65}
With probability $1-o(1)$ we have 
\[
\sigma^{-r}\sum_{i=1}^r\:\sum_{p\in P(\vec{a})}Z_p(\vec{a};q-h_ip)=\left(1+O\left(\frac{1}{(\log_2 x)^5}\right)\right)\frac{u}{\sigma}\frac{x}{2y}\:,
\]
for all but at most $x/(2\log x\:\log_2 x)$ of the primes $q\in Q\cap S(\vec{a})$.\\
\end{lemma}
Before proving Lemma \ref{lima65}, we first confirm that $P/P(a)$ is small with high probability.
\begin{lemma}\label{lima66}
With probability
$$1-O\left(\frac{1}{(\log_2 x)^{10}}\right)\:,$$
$P(\vec{a})$ contains all but
$$O\left(\frac{1}{\log^3 x}\frac{x}{\log x}\right)$$
of the primes $p\in P$. In particular,
$$\mathbb{E}\# P(\vec{a})=\# P\left( 1+O\left(\frac{1}{\log^3 x}\right)\right)\:.$$
\end{lemma}
\begin{proof}
By linearity of expectation and Markov's inequality it suffices that for each $p\in P$
we have $p\in P(\vec{a})$ with probability 
$$1-O\left(\frac{1}{(\log_2 x)^{20}} \right)\:.$$
By Chebyshev's inequality it suffices to show that 
\[
\mathbb{E} X_p(\vec{a})=\mathbb{P}(\tilde{n}_p+h_i p\in S(\vec{a})\ \text{for}\ \text{all}\ i=1,\ldots, r)=\left(1+O\left(\frac{1}{\log_2 x}\right)\right)\sigma^r \tag{6.14}
\]
and
\[
\mathbb{E} X_p(\vec{a})^2=\mathbb{P}(\tilde{n}_p^{(1)}+h_i p,\tilde{n}_p^{(2)}+h_ip\in S(\vec{a})\ \text{for}\ \text{all}\ i=1,\ldots r)=\left(1+O\left(\frac{1}{\log_2 x}\right)\right)\sigma^{2r}\:,\tag{6.15}
\]
where $\tilde{\bold{n}}_p^{(1)}$, $\tilde{\bold{n}}_p^{(2)}$ are independent copies of $\tilde{\bold{n}}_p$ that are also independent of $\vec{\bold{a}}$.\\
To prove the claim (6.14) we first select the value ${n}$ for $\tilde{\bold{n}}_p$
according to the distribution (6.8):
\[
\mathbb{P}(\tilde{\bold{n}}_p=n):=\frac{w(p,n)}{\sum_{n'\in\mathbb{Z}}w(p,n')}\:,
\]
Because of the property $w(p,n)=0$ if $n\not\in \mathcal{G}(p)$ we have with probability 1:
$$n+h_ip\in\mathcal{G}\ \text{for}\ 1\leq i\leq r\:.$$
The relation (6.14) now follows from Lemma \ref{lem61} with $n_i=n+h_ip$, applying the formula for total probability
$$\mathbb{P}(\tilde{n}_p+h_ip\in S(\vec{a}))=\sum_{n}\mathbb{P}(\tilde{n}_p+h_ip\in S(\vec{a})\:|\: \tilde{n}_p=n)\:.$$
A similar application of Lemma \ref{lem61} allows one to write the left hand side of (6.15) as
$$\left(1+O\left(\frac{1}{(\log_2 x)^5}\right)\right) \mathbb{E}\sigma^{\#\{\tilde{\bold{n}}_p^{(l)}+h_ip\::\: i=1,2,\ldots, r, l=1,2\}}\:.$$
From (6.14) we see that the quantity
$$\#\{\tilde{\bold{n}}_p^{(l)}+h_ip\::\: i=1,2,\ldots, r, l=1,2\}$$
is equal to $2r$ with probability 
$$1-O(x^{-1/2-1/6+o(1)})$$ 
and is less than $2r$ otherwise.\\
The claim now follows from $\sigma^{-r}=x^{o(1)}$.
\end{proof}
\textit{Proof of Lemma \ref{lima65}}\\
We first show that replacing $P(\vec{a})$ with $P$ has negligible effect on the sum
with probability $1-o(1)$. Fix $i$ and substitute $n:=q-h_ip$.\\
By Lemma \ref{lem61} we have
\begin{align*}
&\mathbb{E}\sum_n\sigma^{-r}\sum_{p\in P}Z_p(\vec{a};n)=\\ 
&\ \ \ \ \sigma^{-r}\sum_{p\in P}\sum_n \mathbb{P}(\tilde{\bold{n}}_p=n)\mathbb{P}(n+h_ip\in S(\vec{a})\ \text{for}\ j=1,\ldots,r)=\left(1+O\left(\frac{1}{(\log_2 x)^{10}} \right)\right)\# P\:.
\end{align*}
Next by 
$$X_p(\vec{a})=\left(1+O\left(\frac{1}{\log^3 x}\right)\right)\sigma^r$$ 
and Lemma \ref{lem61}
we have
\begin{align*}
\mathbb{E}\sum_r\sigma^{-r}\sum_{p\in P(\vec{a})}Z_p(\vec{a};n)&=\sigma^{-r}\sum_a\mathbb{P}(\vec{\bold{a}}=\vec{a})\sum_{p\in P(\vec{a})}X_p(\vec{a})\\
&=\left(1+O\left(\frac{1}{(\log_2 x)^{10}}\right)\right)\mathbb{E}\#P(\vec{a})\\
&=\left(1+O\left(\frac{1}{\log^3 x}\right)\right)\# P\:.
\end{align*}
Subtracting, we conclude that the left hand side of (6.20) is $O(\# P/\log_2 x)$. The claim then follows from (3.1) and (6.1).\\
By this it suffices to show that 
$$\sigma^{-r}\sum_{i=1}^r \sum_{p\in P}Z_p(\vec{a};q-h_ip)=1+O\left(\frac{1}{\log_2 x}\right)$$
for all but at most $\frac{x}{2\log x \log_2 x}$ primes $q\in Q\cap S(\vec{a})$ one has
\[
\sum_{i=1}^r\sum_{p\in P}Z_p(\vec{a};q-h_ip)=\left(1+O_{\leq}\left(\frac{1}{(\log_2 x)^3} \right)\right)\sigma^{r-1}u\frac{x}{2y} \tag{6.16}
\]
We call a prime $q\in Q$ ``bad'', if $q\in Q\cap S(\vec{a})$, but (6.16) fails. Using
Lemma \ref{lem61} and (6.9) we have
\begin{align*}
&\mathbb{E}\left( \sum_{q\in Q\cap S(\vec{a})}\sum_{i=1}^r\sum_{p\in P}Z_p(\vec{a};q-h_ip)\right)\tag{6.17}\\
&\ \ \ =\sum_{q,i,p}\mathbb{P}(q\in (h_j-h_i)p\in S(\vec{a})\ \text{for}\ \text{all}\ j=1,\ldots,r)
\mathbb{P}(\tilde{n}_p=q-h_ip)
\end{align*}
By the definition of $\mathcal{G}(p)$ we have 
$$\mathbb{P}(q+(h_j-h_i)p\in S(\vec{a}))=0\:,$$
unless $q\in\mathcal{G}(p)$. By Definition \ref{def61} this means that $q+(h_j-h_i)p\in\mathcal{G}$.\\
We may thus apply Lemma \ref{lem61} with 
$$n_j:=(q-h_ip)+h_jp$$ 
and  obtain for all $i$:
$$\mathbb{P}(q+(h_i-h_j)p\in S(\vec{a})\ \text{for}\ \text{all}\ j=1,\ldots,r)=\sigma^r\left(1+O\left(\frac{1}{(\log_2 x)^{10}}\right)\right)\:.$$
With (6.17) we thus obtain
$$\mathbb{E}\left( \sum_{q\in Q\cap S(\vec{a})}\sum_{i=1}^r \sum_{p\in P} Z_p(\vec{a}; q-h_ip)\right)=\left(1+O\left( \frac{1}{(\log_2 x)^{10}}\right)\right)\frac{\sigma y}{\log x}\sigma^{r-1} u\frac{x}{2y}\:.$$
Next we obtain
\begin{align*}
\mathbb{E}\left( \sum_{q\in Q\cap S(\vec{a})}\left( \sum_{i=1}^r \sum_{p\in P} Z_p(\vec{\bold{a}}; q-h_ip)^2\right)\right)&=\sum_{\substack{p_1,p_2,q\\ i_1,i_2}}\mathbb{P}(q+(h_j-h_{i_l})p_l \in S(\vec{a})\ \ \text{for}\ j=1,\ldots, r;\ l=1,2\\  
&\ \ \times \mathbb{P}(\tilde{\bold{n}}_{p_1}^{(1)}=q-h_{i_1}p_1)\mathbb{P}(\tilde{\bold{n}}_{p_2}^{(2)}=q-h_{i_2}p_2)\\
&=\left(1+O\left(\frac{1}{(\log_2 x)^2}\right)\right)\frac{\sigma y}{\log x}\left(\sigma^{r-1} u\frac{x}{2y}\right)^2\:,
\end{align*}
where $(\tilde{{n}}_{p_1}^{(1)})_{p_1\in P}$ and $(\tilde{{n}}_{p_2}^{(2)})_{p_2\in P}$ are independent copies of $(\tilde{{n}}_{p})_{p\in P}$ over $\vec{a}$. In the last step we used the fact that the terms with $p_1=p_2$ contribute negligibly.\\
By Chebyshev's inequality it follows that the number of bad $q$'s is 
\[
\ll \frac{\sigma y}{\log x}\frac{1}{\log_2^2 x}\ll\frac{x}{\log x\log_2^2 x},\ \ \text{with probability}\ 1-O\left(\frac{1}{\log_2 x}\right)\:.\tag{6.18}
\]
\qed\\
We may now prove Theorem \ref{thm4}.\\
The relation (4.4) is actually Corollary \ref{cor6464}. In order to prove (4.3), we 
assume that $\vec{a}$ is good and $q\in Q\cap S(\vec{a})$. Substituting definition
(6.13) into the left hand side of Lemma \ref{lima65} using that $\sigma^{-r}=x^{o(1)}$ and observing that $q=n_i+h_ip$ is only possible if $p\in P(\vec{a})$, we find that 
\begin{align*}
\sigma^{-r}\sum_{i=1}^r\sum_{p\in P(\vec{a})} Z_p(\vec{a};q-h_ip)&=\sigma^{-r}
\sum_{i=1}^r\sum_{p\in P(\vec{a})}X_p(\vec{a})\mathbb{P}(\vec{\bold{n}}_p=q-h_ip\:|\: \vec{\bold{a}}=\vec{a})\\
&=\left( 1+O\left(\frac{1}{(\log_2 x)^2} \right)\right)\sum_{i=1}^r \sum_{p\in P(\vec{a})}\mathbb{P}(\bold{n}_p=q-h_ip\:|\:  \vec{\bold{a}}=\vec{a})\\
&=\left( 1+O\left(\frac{1}{(\log_2 x)^2} \right)\right)\sum_{i=1}^r  \sum_{p\in P}\mathbb{P}(q\in \bold{e}_p(\vec{a})\:|\:\vec{\bold{a}}=\vec{a})\:,
\end{align*}
where
$$e_p(\vec{a})=\{n_p+h_ip\::\: 1\leq i\leq r\}\cap Q\cap S(\vec{a})$$
is as defined in Theorem \ref{thm4}.\\
The fact that $\vec{a}$ is good with probability $1-o(1)$ follows upon noticing that 
$$C:=\frac{u}{\sigma}\frac{x}{2y}\sim \frac{1}{\sigma}\:.$$
This concludes the proof of Theorem \ref{thm4}.\\
\qed\\
It remains to establish Theorem \ref{thm5}. This is the objective of the remaining section of the paper.
\section{Multidimensional Sieve Estimate}
In this section we prove Theorem \ref{thm5} and thus finish the proof of Theorem 
\ref{thm1}. For this purpose we use modifications of the weight functions applied
by K. Ford, B. J. Green, S. Konyagin, J. Maynard and T. Tao in 
\cite{ford2} and J. Maynard \cite{maynard}.
\begin{definition}
An admissible $r-tuple$ is a tuple $(h_1,\ldots,h_r)$ of distinct integers that do not cover all residue classes modulo $p$, for any prime $p$.
\end{definition}
\begin{definition}\label{def72}(see \cite{maynard}, Section 2)\\
Given a set of integers $\mathcal{A}$, a set of primes $P$ and a linear function 
$L(n)=l_1n+l_2$. We write
$$\mathcal{A}(x)=\{n\in \mathcal{A}\::\: x\leq n< 2x\}\:,$$
$$\mathcal{A}(x,q,a)=\{n\in \mathcal{A}(x)\::\: n\equiv a(\bmod\:q)\}\:,$$
$$L(\mathcal{A})=\{L(n)\::\: n\in \mathcal{A}\}$$
$$\Phi_L(q)=\phi\left(\frac{|l_1|q}{\phi(l_1)}\right),\ \ P_{L,\mathcal{A}}(x)=L(\mathcal{A}(x))\cap P,$$
$$P_{L,\mathcal{A}}(x,q,a)=L(\mathcal{A}(x;q,a))\cap P.$$
\end{definition}
\begin{definition}\label{def73}(see \cite{maynard})\\
Let $\mathcal{S}=\{L_1,\ldots,L_g\}$ be a set of distinct linear functions
$$L_i(n)=a_in+b_i\ (1\leq i \leq g)\:,$$ 
with coefficients in the positive integers. $\mathcal{L}$ is called \textbf{admissible} if $\prod_{i=1}^g L_i(n)$ has no fixed prime divisor.\\
We now recall the crucial conditions (7.1) and (7.2) from \cite{maynard}.
\end{definition}
\begin{definition}\label{def74}
Let $\mathcal{A}$, $P$ be as in Definition \ref{def72}, $\mathcal{L}=\{L_1,\ldots,L_g\}$ be an admissible set of integer linear functions, $B$ an integer and quantities $R$, $x$. We assume that the coefficients of $L_i(n):=a_in+b_i\in \mathcal{L}$ satisfy
$|a_i|$, $|b_i|\leq x^\alpha$ and $g=\#\mathcal{L}$ is sufficiently large in terms of the fixed quantities $\theta, \alpha$ and satisfies 
$$g\leq (\log x)^{1/5}B\:,\ \text{with}\ 1\leq B\leq x^\alpha\ \text{and}\ x^{\theta/10}\leq R\leq x^{\theta/3}\:.$$
Finally, we assume that the set $\mathcal{A}$ satisfies 
\[
\sum_{q\leq x^\theta}\max_{a}\left| \#\mathcal{A}(x;q,a)-\frac{\#\mathcal{A}(x)}{q} \right| \ll \frac{\#\mathcal{A}(x)}{(\log x)^{100g^2}}\tag{7.1}
\]
and
\[
\#\mathcal{A}(x;q,a)\ll \frac{\#\mathcal{A}(x)}{q}\:.\tag{7.2}
\]
\end{definition}
\begin{definition}\label{def75}(see \cite{maynard})\\
Let $\mathcal{A}$ be as in Definition \ref{def72}, $\mathcal{L}=\{L_1,\ldots, L_g\}$
be a set of integer linear functions. We define the multiplicative functions 
\[
\omega:=\omega_{\mathcal{L}},\tag{7.3}
\]
\[
\Phi_\omega:=\Phi_{\omega,f}\tag{7.4}
\]
and the singular series 
\[
\mathfrak{G}_D(\mathcal{L})=\prod_{p\nmid D}\left(1-\frac{\omega(p)}{p}\right)\left(1-\frac{1}{p}\right)^{-g}\:.\tag{7.5}
\]
\end{definition}
Since $\mathcal{L}$ is admissible, we have $\omega(p)<p$ for all $p$ and so 
$\Phi_\omega(n)>0$ and $\mathfrak{G}_D(\mathcal{L})>0$ for any integer $D$.
Since $\omega(p)=g$ for all $p\nmid \prod_{i=1}^g a_i\prod_{i,j}(a_ib_j-b_ia_j)$ we see the
product $\mathfrak{G}_D(\mathcal{L})$ converges.
\begin{definition}\label{def76}(see \cite{maynard})\\
Let $\mathcal{L}$ be as in Definition \ref{def73}. Let $B$ be an integer. Let 
$$W=\prod_{\substack{p\leq 2g^2 \\ p\nmid B}}p\:.$$
For each prime $p$ not dividing $B$, let 
$$r_{p,1}(\mathcal{L})<\cdots< r_{p,w_{\mathcal{L}(p)}}(\mathcal{L})$$
be the elements $n$ of $\{1,\ldots, p\}$ for which 
$$p\mid \prod_{i=1}^gL_i(n)\:.$$
If $p$ is also coprime to $W$, then for each $1\leq a\leq \omega_{\mathcal{L}}(p),$ let 
$j_{p,a}=j_{p,a}(\mathcal{L})$ denote the least element of $\{1,2,\ldots, g\}$ such that
$p\mid L_{j_{p,a}}(r_{p,a}(\mathcal{L}))$.\\
Let $\mathcal{D}_g(\mathcal{L})$ denote the set 
\begin{align*}
&\mathcal{D}_g(\mathcal{L}):=\big\{ (d_1,\ldots, d_g)\in \mathbb{N}^g\::\: \mu^2(d_1\cdots d_g)=1,\ (d_1,\ldots, d_g,WB)=1,\\
&\ \ \ \ \ \ \ (d_j,p)=1,\ \text{whenever}\ p\nmid BW\ \text{and}\ j\neq j_{p,1},\ldots, j_{p,\omega_{\mathcal{L}(p)}}\big\}
\end{align*}
Let $x^{\theta/10}\leq R< x^{\theta/3}$ and let $F\::\: \mathbb{R}^g\rightarrow \mathbb{R}$ be a smooth function given by 
$$F(t_1,\ldots, t_g)=\psi\left( \sum_{i=1}^g t_i\right)\prod_{i=1}^g \frac{\psi(t_i/U_g)}{1+T_gt_i}\:,$$
$$T_g=g\log g,\ U_g=g^{-1/2}.$$
Here $\psi\::\: [0,\infty)\rightarrow [0,1)$ is a fixed smooth non-increasing function,
supported on $[0,1)$, which is 1 on $[0,g/10).$\\
For $\bold{r}=(r_1,\ldots, r_g)\in \mathcal{\bold{D}}_g$ let
$$y_{\bold{r}}=\frac{1_{\mathcal{D}_g}(r) W^g B^g}{\Phi(WB)^g}\: \mathfrak{G}_{WB}(\mathcal{L})F\left(\frac{\log r_1}{\log R},\ldots, \frac{\log r_g}{\log R}\right)\:.$$
For $d\in \mathcal{\bold{D}}_g$ let
$$\lambda_\bold{d} =\mu(\bold{d}) \bold{d} \sum_{\bold{d}\mid r}\frac{y_r}{\Phi_\omega(r)}\:.$$
Finally, we set 
$$w_n=w_n(\mathcal{L})=\left(\sum_{d_i\mid L_i(n),\: \forall i} \lambda_{\bold{d}}\right)^2\:.$$
\end{definition}
\textit{Remark.} We find that $\lambda_\bold{d}=0$, if 
$$d=\prod_{i=1}^g d_i>R\:,$$
because of the support of $\psi$ and since $F$ is non-negative.\\ \\
We now borrow the basic theorems for the weights $w_n$ from \cite{maynard}.
\begin{theorem}\label{thm77}
Let $\mathcal{A}$, $w_n$ be as described in Definition \ref{def72} satisfying the 
conditions (7.1) and (7.2). Then we have 
$$\sum_{n\in \mathcal{A}(x)} w_n=\left( 1+O\left(\frac{1}{(\log x)^{1/10}}\right)\right)\frac{B^g}{\phi(B)^g}\mathfrak{G}_B(\mathcal{L})\#\mathcal{A}(x)(\log R)^g I_g(F)\:,$$
where 
$$I_g(F)=\int_0^{+\infty}\cdots\int_0^{+\infty} F^2 \:dt_1\cdots dt_g$$
for a square-integrable function $F\::\: \mathbb{R}^g\rightarrow \mathbb{R}$.\\
The implied constants depend only on $\theta, \alpha$ and the implied constants from (7.1) and (7.2).
\end{theorem}
\begin{proof}
This is Proposition 9.1 of \cite{maynard}.
\end{proof}
\begin{theorem}\label{thm78}
Let $\mathcal{A}$, $w_n$, $\mathcal{L}$ be as described in Definition \ref{def72},
satisfying the conditions (7.1) and (7.2). Let $L(n):=a_mn+b_m\in \mathcal{L}$ satisfy
$L(n)>R$ for $n\in [x,2x)$ and 
\[
\sum_{\substack{q< x^{\theta}\\ (q,B)=1}}\max_{(a,q)=1} \left| \# P_{L,\mathcal{A}}(x,q,a)-\frac{\#P_{L,\mathcal{A}}(x)}{\Phi_L(q)} \right|\ll \frac{\#P_{L,\mathcal{A}}(x)}{(\log x)^{100 g^2}} \tag{7.6}
\]
Then we have
\begin{align*}
\sum_{n\in \mathcal{A}(x)}1_P(L(n))w_n&=\left(1+\left(\frac{1}{(\log x)^{1/10}}\right)\right)\frac{B^{g-1}}{\phi(B)^{g-1}}\mathfrak{G}_B(\mathcal{L})\#P_{L,\mathcal{A}}(x)(\log R)^{g+1}  J_g(F)\prod_{\substack{p \mid a_m\\ p\nmid B}}\frac{p-1}{p}\\
&+O\left(\frac{B^g}{\phi(B)^g}\mathfrak{G}_B(\mathcal{L})\#\mathcal{A}(x)(\log R)^{g+1}I_g(F)\right)\:,
\end{align*}
where 
$$J_g(F)=\int_0^{+\infty}\cdots\int_0^{+\infty} \left(\int_0^{+\infty} F \:dt_g\right)^2\: dt_1\cdots dt_{g-1}\:,$$
for a square-integrable function $F\::\: \mathbb{R}^g\rightarrow \mathbb{R}$.\\
The implied constants depend only on $\theta$, $\alpha$ and the implied constants from (7.1), (7.2) and (7.6).
\end{theorem}
\begin{proof}
This is Proposition 9.2 of \cite{maynard}.
\end{proof}
For the proof of Theorem \ref{thm4} we cannot use the weights $w_n$ directly, but 
we have to modify them to incorporate (for fixed primes $p$) the conditions:
\[
n\equiv 1-(d_p+1)^k(\bmod p),\ d_p\not\equiv -1(\bmod p)\tag{7.7}
\]
and 
\[
n\in \mathcal{G}(p).
\]
We carry out the modification in two steps. In a first step we replace
$w_n=w_n(\mathcal{L})$ by $w^*(p,n)=w^*(p,n,\mathcal{L})$. Here $p$ is a fixed prime with $x/2<p\leq x$.\\
Here we have to be more specific about the set $\mathcal{A}$. We set $\mathcal{A}:=\mathbb{Z}$.
\begin{definition}\label{def79}
Let $w_n$be as in Definition \ref{def72}, $\mathcal{A}=\mathbb{Z}$, $p$ a fixed prime 
with \mbox{$x/2<p\leq x$.} Let also $D=(k-1,p)$. We set
\begin{eqnarray}
\nonumber\\
w^*(p,n)&=&\left\{ 
  \begin{array}{l l}
    Dw_n,\ \  \text{if there is $d_p\in\mathbb{Z}$ with $n\equiv 1-(d_p+1)^k(\bmod p),\ d_p\not\equiv -1(\bmod p)$\ \ \ \ \text{(*)}}\vspace{2mm}\ \ \\ 
    0,\ \  \text{otherwise.}\vspace{2mm}\\ 
  \end{array} \right.
\nonumber
\end{eqnarray}
\end{definition}
We first express the solvability of (*) by the use of Dirichlet-characters.
\begin{lemma}\label{lima710}
Let $p$ be a prime number. Let  $D=(p-1,k)$, $\chi_0$ the principal character $\bmod\:D$. There are $D-1$ non-principal characters $\chi_1,\ldots,\chi_{D-1}$ $\bmod\: D$, such that for all $n\in\mathbb{Z}$ we have:
\begin{eqnarray}
\nonumber\\
\frac{1}{D}\sum_{l=0}^{D-1}\chi_l(1-n)&=&\left\{ 
  \begin{array}{l l}
    1,\ \  \text{if $n\equiv 1-c^k(\bmod\: p)$ is solvable with $p\nmid c$}\vspace{2mm}\ \ \\ 
    0,\ \  \text{otherwise.}\vspace{2mm}\\ 
  \end{array} \right.
\nonumber
\end{eqnarray}
\end{lemma}
\begin{proof}
Let $\rho$ be a primitive root $\bmod\: p$, 
$$1-n\equiv \rho^s(\bmod\: p),\ \ 0\leq s\leq p-2\:.$$
Setting 
$$c\equiv  \rho^y(\bmod\: p),$$
we see that the congruence 
\[
c^k\equiv 1-n(\bmod\: p)\tag{7.9}
\]
is solvable if and only if
\[
ky\equiv s(\bmod\: p-1) \tag{7.10}
\]
has a solution $y$.\\
By the theory of linear congruences, this is equivalent to $D\mid s$. We have
\begin{eqnarray}
\nonumber\\
\frac{1}{D}\sum_{l=0}^{D-1}e\left(\frac{ls}{D}\right)&=&\left\{ 
  \begin{array}{l l}
    1,\ \  \text{if $D\mid s$}\vspace{2mm}\ \ \\ 
    0,\ \  \text{otherwise.}\vspace{2mm}\\ 
  \end{array} \right.
\nonumber
\end{eqnarray}
We now define the Dirichlet-character 
$$\chi_l(0\leq l\leq D-1)\ \ \text{by}\ \ \chi_l(-n)=e\left(\frac{ls}{D}\right)$$
and obtain the claim of Lemma \ref{lima710}.
\end{proof}
\begin{theorem}\label{thm711}
Let $p,w^*(p,n)$, $D$ as in Definition \ref{def79}, $\mathcal{A}:=\mathbb{Z}$. Then
we have
$$\sum_{n\in \mathcal{A}(x)}w^*(p,n)=\left(1+O\left(\frac{1}{(\log x)^{1/10}}\right)\right)\frac{B^k}{\phi(B)^k}\mathfrak{G}_B(\mathcal{L})\mathcal{A}(x)(\log R)^k I_k(F).$$
\end{theorem}
\begin{proof}
By Lemma \ref{lima710} we have
$$\sum_{n\in \mathcal{A}(x)}w^*(p,n)=\sum_{l=0}^{D-1}\sum_{n\in \mathcal{A}(x)}w_n\chi_l(1-n)$$
The sum belonging to the principal character 
$$\chi_0\::\: \sum_{n\in \mathcal{A}(x)}w_n\chi_0(1-n)$$ 
differs from the sum 
$$\sum_{n\in \mathcal{A}(x)}w_n$$ 
only by $O(x^{1/2})$,
since there are only $\frac{|\mathcal{A}(x)|}{p}$ terms with $n\equiv 1(\bmod p)$, each of them 
has size at most $x^{1/3}$. We therefore have
\[
\sum_{n\in \mathcal{A}(x)}w_n\chi_0(1-n)=\sum_{n\in \mathcal{A}(x)} w_n+O(x^{1/2})\:.\tag{7.11}
\]
Let now $1\leq l\leq D-1$. Here we closely follow the proof of Proposition 9.1 of 
\cite{maynard}. We split the sum into residue classes $n\equiv v_0(\bmod\: W)$. We recall that
$$W=\prod_{\substack{p\leq 2g^2\\ p\nmid B}}p<\exp \left(\left( \log x\right)^{2/5}\right)\:.$$
If $\left(\prod_{i=1}^g L_i(v_0),W\right)=1$, then we have $w_n=0$ and so we restrict our attention to $v_0$ with $\left(\prod_{i=1}^g L_i(v_0),W\right)=1$. We substitute the definition of $w_n$, expand the square and swap order of summation. This gives
$$\sum_{n\in\mathcal{A}(x) }\chi_l(1-n)=\sum_{v_0(\bmod\: W)}\sum_{\bold{d},\bold{e}\in D_g}\lambda_{{\bold{d}}}\lambda_{{\bold{e}}} \sum_{\substack{n\in \mathcal{A}(x)\\ n\equiv v_0(\bmod\: W)\\ [d_i,e_i]\mid L_i(n),\:\forall i}}\chi_l(1-n)\:.$$
The congruence conditions in the inner sum may be combined via the Chinese Remainder Theorem by a single congruence condition:
$$1-n\equiv c(\bmod\: v),\ \ \text{where}\ v=W[\bold{d},\bold{e}]\:,$$
where $[\cdot,\cdot]$ stands for the least common multiple.\\
There are $w\leq v$ Dirichlet characters $\psi_1,\ldots, \psi_w$ $(\bmod\: w)$, such that
$$1-n\equiv c(\bmod\: v)\ \ \text{if and only if}\ \ \frac{1}{w}\sum_{l=1}^w\overline{\psi(c)}\psi_l(1-n)=1\:.$$
We thus may write
$$\sum_{\substack{n\in\mathcal{A}(x)\\ n\equiv v_0(\bmod\: w)\\ [d_i,e_i]\mid L_i(n),\:\forall i}}\chi_l(1-n)=A\sum_{l=1}^z \sum_{n\in \mathcal{I}}\xi_l(1-n)   $$
with a suitable absolute constant $A$, an interval $\mathcal{I}$ of length $|\mathcal{I}|\leq x(\log x)^2$ and the $O(v)$ non-principal Dirichlet characters
$\xi_{j,l}=\chi_j\psi_l$ of conductor $\geq p$ and modulus $\leq xv$.
By the P\'olya-Vinogradov bound we obtain:
\[
\sum_{\substack{n\in\mathcal{A}(x)\\ -n\equiv v_0(\bmod\: w)\\ [d_i,e_i]\mid L_i(n),\:\forall i}}\chi(1-n)\ll x^{1/2}v\:. \tag{7.12}
\]
The claim of Theorem \ref{thm711} now follows from (7.11) and (7.12).\\
\end{proof}
As a preparation for the proof of Theorem \ref{thm713}, which is a modification of 
Proposition 9.2 of \cite{maynard} we state a Lemma on character sums over shifted
primes.
\begin{lemma}\label{lima712}
Let $\chi$ be a Dirichlet-character $(\bmod\: q)$. Then for $N\leq q^{16/9}$ we have
$$\sum_{n\leq N} \Lambda(n)\chi(n+a)\leq (N^{7/8}q^{1/9}+N^{33/32}q^{-1/18})q^{o(1)}\:.$$
\end{lemma}
\begin{proof}
This is Theorem 1 of \cite{fried}.
\end{proof}
\begin{theorem}\label{thm713}
Let $p$, $w^*(p,n)$ be as in Definition \ref{def79}, $\mathcal{A}=\mathbb{Z}$. Let 
$$L(n):=a_mn+b_m\in \mathcal{L}$$ satisfy $L(n)>R$ for $n\in [x,2x]$ and 
$$\sum_{\substack{q\leq x^\theta \\ (q,B)=1}}\max_{L(a,q)=1}\left| \#P_{L,\mathcal{A}}(x;q,a)-\frac{\#P_{L,\mathcal{A}}(x)}{\Phi_L(q)} \right| \ll\frac{\#P_{L,\mathcal{A}}(x)}{(\log x)^{100 g^2}}\:. $$
Then we have for sufficiently small $\theta$:
\begin{align*}
\sum_{n\in \mathcal{A}(x)}1_P(L(n))w^*(p,n)&=\left(1+O\left(\frac{1}{(\log x)^{1/10}}\right)\right)\frac{B^{g-1}}{\phi(B)^{g-1}}\mathfrak{G}_B(\mathcal{L})\#P_{L,\mathcal{A}}(x)(\log R)^{g+1}  J_g(F)\prod_{\substack{p\mid a_m\\ p\nmid B}}\frac{p-1}{p}\\
&+O\left(\frac{B^g}{\phi(B)^g}\mathfrak{G}_B(\mathcal{L})\#\mathcal{A}(x)(\log R)^{g-1}I_g(F)\right)\:.
\end{align*}
\end{theorem}
\begin{proof}
By Lemma \ref{lima710} we have
$$\sum_{n\in\mathcal{A}(x)}1_P(L(n))w^*(p,n)=\frac{1}{D}\sum_{l=0}^{D-1}\sum_{n\in \mathcal{A}(x)}1_P(L(n))w_n\chi_l(1-n)\:.$$
The sum belonging to the principal character $\chi_0$ differs from the sum 
$$\sum_{n\in \mathcal{A}(x)}1_P(L(n))w_n$$ 
only by $O(\# \mathcal{A}(x)p^{-1})$ and thus by \cite{maynard}, Proposition 9.2,
we have 
\begin{align*}
\sum_{n\in \mathcal{A}(x)}1_P((L(n)) w^*(p,n)\chi_0(1-n)&=\left(1+O\left(\frac{1}{(\log x)^{1/10}}\right)\right)\frac{B^{g-1}}{\phi(B)^{g-1}}\mathfrak{G}_B(\mathcal{L}) \tag{7.13}\\
&\times \# P_{L,\mathcal{A}}(x)(\log R)^{g+1} \mathcal{I}_g(F)\prod_{p\mid a_m}\frac{p-1}{p} \\
&+O\left(\frac{B^g}{\phi(B)}\mathfrak{G}_B(\mathcal{L})\#\mathcal{A}(x)(\log R)^{g-1} I_g(F)\right)  
\end{align*}
For $1\leq l\leq D-1$ we follow closely the proof of Proposition 9.2 in \cite{maynard}.\\
We again split the sum into residue classes $n\equiv v_0(\bmod\: W)$.\\
If $(\prod_{i=1}^g L_i(v_0), W)>1$, then we have $w_n=0$ and so we restrict  our
attention to $v_0$ with $(\prod_{i=1}^g L_i(v_0), W)=1$.
We substitute the definition of $w_n$, expand the square and swap order of summation. Setting $\tilde{n}=n-1$, we obtain
\[
\sum_{\substack{n\in\mathcal{A}(x)\\ n\equiv v_0(\bmod\: W) }}1_P((L(n))w_n\chi_l(1-n)=\sum_{\bold{d},\bold{e}}\lambda_\bold{d} \lambda_\bold{e} \sum_{\substack{n\in\mathcal{A}(x)\\ n\equiv v_0(\bmod\: W) \\ [d_i,e_i]\mid L_i(n),\:\forall i}}1_P((L(n))\chi_l(1-n)\:.   \tag{7.14}
\]
If $\tilde{n}$ runs through the arithmetic progression $\tilde{n}=Wh+v_0$ $(h\in I_0)$, then also $L(\tilde{n}+1)$ runs through an arithmetic progression
$$L(\tilde{n}+1)=a_mWh+a_m(v_0+1)+b\:.$$
Thus, we have 
$$\sum_{\substack{n\in\mathcal{A}(x)\\ n\equiv v_0(\bmod\: W) }}1_P((L(n)) \chi_l(1-n)=\sum_{\substack{\tilde{p}\equiv a_m(v_0+1)+b(\bmod\: a_mW)\\ \tilde{p}\ \text{prime},\:\tilde{p}\in I}}\chi_l(\tilde{p}+a_m(v_0+1)+b)\:.$$
Also the condition $\tilde{p}\equiv a_m(v_0+1)+b(\bmod\: a_mW)$ may be expressed with the 
help of Dirichlet characters $\omega_1,\ldots, \omega_{\phi(|a_mW|)}(\bmod\: |a_mW|)$. Theorem \ref{thm713} thus follows from (7.13) and Lemma \ref{lima712}.
\end{proof}
For the definition of the weight $w(p,n)$, whose existence is claimed in \mbox{Theorem \ref{thm5}} we now have to be more specific about the set 
$\mathcal{L}$ of linear forms.
\begin{definition}\label{def714}
Let the tuple $(h_1,\ldots,h_r)$ be as in Definition \ref{def76}. For $p\in P$ and $n\in\mathbb{Z}$ let $\mathcal{L}_p$ be the (ordered) collection of linear forms $n\mapsto n+h_ip$ for $i=1,\ldots, r$ and set 
\begin{eqnarray}
\nonumber\\
w(p,n)&=&\left\{ 
  \begin{array}{l l}
    w^*(p,n,\mathcal{L}_p), &\text{if $n\in \mathcal{G}(p)$}\vspace{2mm}\\ 
    0, &\text{otherwise.}\vspace{2mm}\\ 
  \end{array} \right.
\nonumber
\end{eqnarray}
\end{definition}
In the sequel we now show that in the sums 
$$\sum_{n\in\mathbb{Z}} w(p,n)\ \ \text{resp.}\ \ \sum_{p\in P}w(p,q-h_ip)\:,$$
appearing in (6.4) resp. (6.5) of Theorem \ref{thm5} the function $w(p.\cdot)$ may
be replaced by the function $w^*(p,\cdot,\mathcal{L}_p)$ with a negligible error.\\
Since these sums have been treated in Theorem \ref{thm711} resp. Theorem
\ref{thm713} this will essentially conclude the proof of Theorem \ref{thm5} and thus of Theorem \ref{thm1}.
\begin{lemma}\label{lima715}
We have
$$\sum_{\substack{n\in \mathcal{A}(x)\\ n\not\in \mathcal{G}(p)}}w^*(p,n,\mathcal{L}_p)\leq \sum_{\substack{n\in \mathcal{A}(x)\\ n\not\in \mathcal{G}(p)}} w_n(\mathcal{L}_p)\:.$$
\end{lemma}
\begin{proof}
This follows immediately from the definition of $w_n$ and $w^*$.
\end{proof}
\begin{definition}\label{def716}
Let $(h_1,\ldots,h_r)$ be an admissible $r$-tuple, $p\in (x/2,x)$. For $n\in\mathbb{Z}$, $1\leq i,l\leq r$ let
$$\tilde{n}=\tilde{n}(n,i,l,p)=n+(h_i-h_l)p$$
$$\mathcal{A}(i,l,p)=\sum_{n\::\: \tilde{n}\not\in \mathcal{G}} w_n(\mathcal{L}_p)$$
Let 
$$\sum(i,l,p):=\sum_{n\in \mathcal{A}(x)}w_n(\mathcal{L}_p) (r(\tilde{n},u)-r^*(u))^2$$
$$\sum(i,l,p,j):=\sum_{n\in \mathcal{A}(x)}w_n(\mathcal{L}_p) r(\tilde{n},u)^j \ \ (j\in\mathbb{N}_0)$$
\end{definition}
\begin{lemma}\label{lima717}
$$\sum_{\substack{n\in\mathcal{A} \\ \tilde{n}\not\in \mathcal{G}(p)}}w_n(\mathcal{L}_p)=\sum_{1\leq i,l\leq r}\mathcal{A}(i,l,p)\:.$$
\end{lemma}
\begin{proof}
This follows immediately from Definitions \ref{def61} and \ref{def716}.
\end{proof}
\begin{lemma}\label{lima718}
Let $\mathcal{A}$, $w_n$ be as in Definition \ref{def72}, $\mathcal{L}_p$ as in Definition \ref{def714}. Let $j\in\{1,2\}$. Then
$$\sum(i,l,p,j)=\left(1+O\left(\frac{1}{(\log x)^{1/10}}\right)\right)\frac{B^g}{\phi(B)^g}\mathfrak{G}_B(\mathcal{L}_p)\#\mathcal{A}(x)(\log R)^gI_g(F) r^*(u)^j\:. $$
\end{lemma}
\begin{proof}
We only give the proof for the hardest case $j=2$ and shortly indicate the proof for 
$j=1$.
\begin{align*}
\sum(i,l,p,2)&=\sum_{n\in\mathcal{A}(x)}w_n(\mathcal{L}_p)r(\tilde{n},u)^2\\
&=\sum_{n\in\mathcal{A}(x)}w_n(\mathcal{L}_p) \left(\frac{1}{d(u)^2}\sum_{\substack{s_1\in S_u,\ c_{s_1}\in \{0,1,\ldots, s_1-2\}\\ \tilde{n}\equiv 1-(c_{s_1}+1)^k(\bmod s_1)}} s_1^{-1}\right)\left(\sum_{\substack{s_2\in S_u,\ c_{s_2}\in \{0,1,\ldots, s_2-2\}\\ \tilde{n}\equiv 1-(c_{s_2}+1)^k(\bmod s_2)}} s_2^{-1}\right)\\
&=\frac{1}{d(u)^2}\sum_{s_1,s_2\in S_u}s_1^{-1}s_2^{-1}\sum_{c_{s_1}=1}^{s_1-2}\sum_{c_{s_2}=1}^{s_2-2}\ \sum_{\substack{n\equiv 1-(c_{s_1}+1)^k+(h_l-h_i)p(\bmod s_1)\\ n\equiv 1-(c_{s_2}+1)^k+(h_l-h_i)p(\bmod s_2)}}w_n(\mathcal{L}_p)
\end{align*}
In the inner sum we only deal with the case $s_1\neq s_2$, the case $s_1=s_2$ giving a negligible contribution. The inner sum is non-empty if and only if the 
system 
\begin{eqnarray}
\nonumber\\
(*)\ \ \ \left\{ 
  \begin{array}{l l}
    \tilde{n}\equiv 1-(c_{s_1}+1)^k(\bmod s_1)\vspace{2mm}\\ 
   \tilde{n}\equiv 1-(c_{s_2}+1)^k(\bmod s_2)\vspace{2mm}\\ 
  \end{array} \right.
\nonumber
\end{eqnarray}
is solvable. In this case (*) is equivalent to a single congruence 
$$n\equiv e+(h_l-h_i)p(\bmod\: s_1s_2)\:,$$
where $e=e(s_1,s_2,c_1,c_2)$ is uniquely determined by the system (*) and 
$$0\leq e\leq s_1s_2-1\:.$$
We apply Theorem \ref{thm77} with $B$ independent of $s_1$, $s_2$ and with
$$\mathcal{A}=\mathcal{A}^{(s_1,s_2)}=\{n\::\: x/2<n\leq x,\ n\equiv e+(h_l-h_i)p(\bmod s_1s_2)\}\:.$$
We have 
$$\# \mathcal{A}^{(s_1,s_2)}(x)=s_1^{-1}s_2^{-1}\mathcal{A}(x)+O(1)$$
and obtain 
\[
\sum_{n\in\mathcal{A}(x)}w_n(\mathcal{L}_p)r(\tilde{n},u)^2=\left(1+O\left(\frac{1}{(\log x)^{1/10}}\right)\right)\frac{B^g}{\phi(B)^g}\mathfrak{G}_B(\mathcal{L}_p)\#\mathcal{A}(x)(\log R)^gI_g(F) r^*(u)^2 \tag{7.15}
\]
This proves the claim for $j=2$. The proof of the case $j=1$ is analogous but simpler, since there is only the single variable of summation $s_1$.
\end{proof}
\begin{lemma}\label{lima719}
Let the conditions be as in Lemmas \ref{lima715}, \ref{lima717}, \ref{lima718}. Then we have
$$\sum(i,l,p)\ll \frac{B^g}{\phi(B)^g}\mathfrak{G}_B(\mathcal{L}_p)\#\mathcal{A}(x)(\log R)^gI_g(F) r^*(u)^2(\log x)^{1/8}.$$
\end{lemma}
\begin{proof}
This follows from Lemma \ref{lima718} and the identity
$$\sum(i,l,p)=\sum(i,l,p,2)-2r^*(u)\sum(i,l,p,1)+r^*(u)^2\sum(i,l,p,0)\:.$$
\end{proof}
\begin{theorem}\label{thm720}
Let the conditions be as in the previous lemmas. For sufficiently small $\eta_0$ we
have:
$$\sum_{\substack{n\in\mathcal{A}(x)\\ n\not\in \mathcal{G}(p)}}w_n(\mathcal{L}_p)\ll
 \frac{B^g}{\phi(B)^g}\mathfrak{G}_B(\mathcal{L}_p)\#\mathcal{A}(x)(\log R)^gI_g(F) (\log x)^{-1/100}$$
\end{theorem}
\begin{proof}
Let $1\leq i,R\leq r$. By Definition \ref{def61} we have
$$\tilde{n}=n+(h_i-h_l)p\not\in \mathcal{G}$$
which yields
$$|r(\tilde{n},u)-r^*(u)|\geq r^*(u)(\log x)^{-1/40}\:.$$
Thus
\begin{align*}
r^*(u)^2(\log x)^{-1/20}&\sum_{n\in\mathcal{A}(x)\::\: n+ (h_i-h_l)p\not\in \mathcal{G}(p)}w_n(\mathcal{L}_p)\\
&\leq \sum_{n\in \mathcal{A}(x)} w_n(\mathcal{L}_p)(r(\tilde{n},u)-r^*(u))^2\\
&\leq \frac{B^g}{\phi(B)^g}\mathfrak{G}_B(\mathcal{L}_p)\#\mathcal{A}(x)(\log R)^g I_g(F) (\log x)^{-1/10}r^*(u)^2
\end{align*}
and therefore
$$\sum_{n\in\mathcal{A}(x)\::\: n+ (h_i-h_l)p\in \mathcal{G}(p)}w_n(\mathcal{L}_p) \leq \frac{B^g}{\phi(B)^g}\mathfrak{G}_B(\mathcal{L}_p)\#\mathcal{A}(x)(\log R)^g I_g(F) (\log x)^{-1/20} \:.$$
The claim of Theorem \ref{thm720} follows by summation over all pairs $(i,l)$, if
$\eta_0$ is sufficiently small.
\end{proof}
We now investigate the sum in (6.5) of Theorem \ref{thm5}.
\begin{definition}\label{def721}
Let $x/2<p\leq x$, $L(u)=n+h_fp$. Let $L\in \mathcal{L}_p$; $1\leq i,l\leq r$. Then we
define
$$\mathcal{C}(i,l,p):=\sum_{n\in \mathcal{A}(x),\tilde{n}\not\in \mathcal{G}}1_P(L(n))w_n(\mathcal{L}_p)$$
$$\Omega(i,l,p):=\sum_{n\in \mathcal{A}(x)}1_P(L(n))w_n(\mathcal{L}_p)(r(\tilde{n},u)-r^*(u))^2$$
$$\Omega(i,l,p,j):=\sum_{n\in \mathcal{A}(x)}1_P(L(n))w_n(\mathcal{L}_p)r(\tilde{n},u)^j\:.$$
\end{definition}
\begin{lemma}\label{lima722}
Let $L, i, l, r, p$ be as in Definition \ref{def721}. Let $j\in \{1,2\}$. Then we have
\begin{align*}
\Omega(i,l,p,j)&= \frac{B^{g-1}}{\phi(B)^{g-1}}\mathfrak{G}_B(\mathcal{L}_p)\#\mathcal{P}_{L,\mathcal{A}}(x)(\log R)^{g+1}\\
&\ \ \ \times J_g(F)r^*(u)^j(1+O((\log x)^{-1/10}))\\
&\ \ \ \ + O\left(  \frac{B^g}{\phi(B)^g}\mathfrak{G}_B(\mathcal{L}_p)\#\mathcal{A}(x)(\log R)^{g-1} J_g(F) r^*(u)^j  \right)
\end{align*}
\end{lemma}
\begin{proof}
We only give the proof for the hardest case $j=2$. The case $j=1$ is analogous but simpler. We have
\begin{align*}
\Omega(i,l,p,2)&=\sum_{n\in\mathcal{A}(x)}1_P(L(n))w_n(\mathcal{L}_p)r(\tilde{n},u)^2  \\
&=\sum_{n\in\mathcal{A}(x)}1_P(L(n))w_n(\mathcal{L}_p) \left(\frac{1}{d(u)^2}\sum_{\substack{s_1\in S_u,\ c_{s_1}\in\{1,\ldots, s_1-2\}\\ \tilde{n}\equiv1-(c_{s_1}+1)^k(\bmod s_1)}}s_1^{-1}\right)\\
&\ \ \ \ \times \left(\sum_{\substack{s_2\in S_u,\ c_{s_2}\in\{1,\ldots, s_2-2\}\\ \tilde{n}\equiv1-(c_{s_2}+1)^k(\bmod s_2)}}s_2^{-1}\right)\\
&= \frac{1}{d(u)^2}\sum_{s_1,s_2\in S_u} s_1^{-1}s_2^{-1}\sum_{c_{s_1}=1}^{s_1-2}\sum_{c_{s_2}=1}^{s_2-2}\ \sum_{\substack{n\equiv 1-(c_{s_1}+1)^k+(h_l-h_i)p(\bmod s_1)\\ n\equiv 1-(c_{s_2}+1)^k+(h_l-h_i)p(\bmod s_2)\\ n\in\mathcal{A}(x)}} 1_P(L(n))w_n(\mathcal{L}_p)\:.
\end{align*}
We deal only with the case $s_1\neq s_2$ for the inner sum, the case $s_1=s_2$
giving a negligible contribution. The inner sum is non-empty if and only if the system
\begin{eqnarray}
\nonumber\\
(*)\ \ \ \left\{ 
  \begin{array}{l l}
    {n}\equiv 1-(c_{s_1}+1)^k(\bmod s_1)\vspace{2mm}\\ 
   {n}\equiv 1-(c_{s_2}+1)^k(\bmod s_2)\vspace{2mm}\\ 
  \end{array} \right.
\nonumber
\end{eqnarray}
is solvable.\\
In this case the system is equivalent to a single congruence $n\equiv e(s_1,s_2,c_1, c_2)$ is uniquely determined by the system (*) and $0\leq e\leq s_1s_2-1$. The inner sum then takes the form
$$\sum_{\substack{n\equiv e(\bmod s_1s_2)\\ n\in \mathcal{A}(x)}}1_P(L(n))w_n(\mathcal{L}_p)\:.$$
By the substitution $n=ms+e$, we get
$$L(n)=L^*(m,s)=ms+e+h_fp\:.$$
We set $\mathcal{L}_p=\{L_{h_i}\}$, where $L_{h_i}(n)=n+h_ip$ is replaced by the set $\mathcal{L}_{p,s}=\{L_{h_i,s}\}$, where 
$$L_{h_i,s}(m)=ms+e+(h_i+h_f)p\:. $$
We thus have
\begin{align*}
\sum(s)&:=\sum_{\substack{n\equiv e(\bmod s)\\ n\in \mathcal{A}(x)}}1_P(L(n))w_n(\mathcal{L}_p)\\
&=\sum_{m\in\mathcal{A}\left(\frac{x}{s}\right) }w_m(\mathcal{L}_{p,s})1_P(L^*(m,s))+O(1)\:.
\end{align*}
We apply Theorem \ref{thm78} with $\mathcal{A}:=\mathbb{N}$, $x/s$ instead of $x$,
$L(\cdot)=L^*(\cdot,s)$, $\mathcal{L}=\mathcal{L}_{p,s}$.\\
We have 
$$\mathfrak{G}_B(\mathcal{L}_p)=\mathfrak{G}_B(\mathcal{L}_{p,s})\left(1+O\left(\frac{1}{\log x}\right)\right)\:.$$
From Bombieri's Theorem it can easily be seen that the conditions (7.6) are satisfied 
for all $s$ with the possible exception of $s\in \mathcal{E}$, $\mathcal{E}$ being an exceptional set, satisfying
$$\sum_{s\in \mathcal{E}}s^{-1}\ll (\log x)^{-4}\:.$$
For $s\in \mathcal{E}$ we use the trivial bound $1_P(L^*(m,s))=O(1)$. Thus, we obtain the claim of Lemma \ref{lima722} for the case $j=2$.\\
The proof for $j=1$ is analogous but simpler, since we have only to sum over the single variable $s_1$.
\end{proof}
\begin{lemma}\label{lima723}
Let $i,l,p$ as in Definition \ref{def721}. We have
\begin{align*}
\Omega(i,l,p)&=O\left( \frac{B^{g-1}}{\phi(B)^{g-1}}|\mathfrak{G}_B(\mathcal{L}_p)|\#\mathcal{P}_{L,\mathcal{A}}(x)(\log R)^{g+1} J_g(F)r^*(u)^2(\log x)^{-1/10}\right)\\
&\ \ \ \ + O\left(  \frac{B^g}{\phi(B)^g}|\mathfrak{G}_B(\mathcal{L}_p)|\#\mathcal{A}(x)(\log R)^{g-1} J_g(F) r^*(u)^2  \right)
\end{align*}
\end{lemma}
\begin{proof}
By Definition \ref{def721} we have 
$$\Omega(i,l,p)=\Omega(i,l,p,2)-2r^*(u)\Omega(i,l,p,1)+r^*(u)^2\Omega(i,l,p,0)\:.$$
Lemma \ref{lima723} now follows by applying Lemma \ref{lima722}.
\end{proof}
\begin{theorem}\label{thm724}
Let $p$, $L(n)$ be as in Definition \ref{def721}. Then we have
$$\sum_{\substack{n\in \mathcal{A}(x)\\ n\not\in \mathcal{G}(p)}}1_P(L(n))w_n(\mathcal{L}_p)\ll  \frac{B^{g-1}}{\phi(B)^{g-1}}\mathfrak{G}_B(\mathcal{L}_p)\#\mathcal{P}_{L,\mathcal{A}}(x)(\log R)^{g+1} J_g(F)(\log x)^{-1/100} $$
\end{theorem}
\begin{proof}
Let $1\leq i,l\leq r$. By Definition \ref{def61} we have:
$$\tilde{n}=n+(h_i-h_l)p\in \mathcal{G}\:.$$
It follows that
$$|r(\tilde{n},u)-r^*(u)|\geq r^*(u)(\log x)^{-1/40}\:.$$
Thus
\begin{align*}
&r^*(u)^2(\log x)^{-1/20}\sum_{n\in \mathcal{A}(x)\::\: \tilde{n}\not\in \mathcal{G}(p)}1_P(L(u))w_n(\mathcal{L}_p) \\
&\ \ \  \ll \frac{B^{g-1}}{\phi(B)^{g-1}}\mathfrak{G}_B(\mathcal{L}_p)\#\mathcal{P}_{L,\mathcal{A}}(x)(\log R)^{g+1} J_g(F)r^*(u)^2(\log x)^{-1/16}\\
 &\ \ \ \ \ \ \ + \frac{B^{g}}{\phi(B)^{g}}\mathfrak{G}_B(\mathcal{L}_p)\#\mathcal{A}(x)(\log R)^{g-1} J_g(F)r^*(u)^2\tag{7.15}
\end{align*}
The second term is absorbed in the first one, since by Definition \ref{def76}:
$$x^{\theta/10}\leq R\leq x^{\theta/3}$$ 
and thus 
$$\log R \asymp \log x\:.$$
Therefore
$$\mathcal{C}(i,l,p)\ll   \frac{B^{g-1}}{\phi(B)^{g-1}}\mathfrak{G}_B(\mathcal{L}_p)\#\mathcal{P}_{L,\mathcal{A}}(\log R)^{g+1} J_g(F)(\log x)^{-1/20}$$
The claim of the Theorem \ref{thm724} now follows by summing over all pairs $(i,j)$.
\end{proof}
We now can conclude the proof of Theorem \ref{thm5} and therefore also the proof of Theorem \ref{thm1}.\\
By Theorems \ref{thm711}, \ref{thm713}, \ref{thm720}, \ref{thm724} we have
\[
\sum_{n\in \mathcal{A}(x)}w(p,n)=\left(1+O\left(\frac{1}{(\log x)^{1/100}}\right)\right)\sum_{n\in \mathcal{A}(x)}w_n(\mathcal{L}_p) \tag{7.16}
\]
and 
\[
\sum_{n\in \mathcal{A}(x)}1_P(L(n))w(p,n)=\left(1+O\left(\frac{1}{(\log x)^{1/100}}\right)\right)\sum_{n\in \mathcal{A}(x)}1_P(L(n))w_n(\mathcal{L}_p)\tag{7.17}
\]
The deduction of the equations (6.4) and (6.5) of Theorem \ref{thm5} can thus be deduced from results on the sums on the right hand side of equations (7.16) 
and (7.17). These results are obtained in Section 8 (Verification of sieve estimates) of \cite{ford2}.\\ \\
\noindent\textbf{Acknowledgements}. The second author (M. Th. R.) wishes to 
express his gratitude to Professors E. Kowalski and P. Sarnak for providing him the opportunity to conduct postdoctoral research at the Department of Mathematics of Princeton University during the academic year 2014-2015, where the present work was initiated. He would also like to express his gratitude to Professors A. Nikeghbali and P.-O. Dehaye for providing him the opportunity to currently conduct postdoctoral research at the Institute of Mathematics of the University of Zurich. The completion of his work was supported by the
SNF grant: SNF PP00P2\textunderscore138906 of the Swiss National Science Foundation.

\end{document}